\documentclass[a4paper,11pt]{article}
\usepackage[latin1]{inputenc}
\usepackage[T1]{fontenc}
\usepackage{amsmath,amssymb,amstext}
\usepackage{theorem}
\usepackage{verbatim}
\usepackage{a4wide}

\newtheorem{theorem}{Theorem}[section]

\newtheorem{corollary}[theorem]{Corollary}

\newtheorem{definition}[theorem]{Definition}

\newtheorem{lemma}[theorem]{Lemma}

\newtheorem{proposition}[theorem]{Proposition}
\newtheorem{remark}[theorem]{Remark}

\newenvironment{proof}[1][Proof]{\noindent\textbf{#1.} }{\ \rule{0.5em}{0.5em}}
\newcommand{\R}{\mathbb{R}}
\newcommand{\Rd}{\mathbb{R}^{d}}
\newcommand{\C}{\mathbb{C}}
\newcommand{\E}{\mathbb{E}}
\newcommand{\p}{\mathbb{P}}

\newcommand{\N}{\mathcal{N}}
\newcommand{\Nd}{\mathcal{N}_{d}}
\newcommand{\h}{\mathfrak{H}}
\newcommand{\smd}{\sum\limits_{i=1}^{d}}

\newcommand{\pxi}{\cfrac{\partial}{\partial x_i}}
\newcommand{\pxj}{\cfrac{\partial}{\partial x_j}}
\newcommand{\pxk}{\cfrac{\partial}{\partial x_k}}
\newcommand{\pxij}{\cfrac{\partial^2}{\partial x_i \partial x_j}}

\newcommand{\law}{\overset{(law)}{=}}
\newcommand{\Hess}{{\rm Hess }\,}
\newcommand{\1}{\mathbf{1}}

\title{ \bf Multi-dimensional Gaussian fluctuations\\ on the Poisson space}
\author{Giovanni Peccati\footnote{Facult\'{e} des Sciences, de la Technologie
et de la Communication; UR en Math\'{e}matiques. 6, rue Richard
Coudenhove-Kalergi, L-1359 Luxembourg. Email: \texttt{giovanni.peccati@gmail.com}} \ \ and \ Cengbo Zheng\footnote{Equipe Modal'X, Universit\'{e} Paris Ouest -- Nanterre la D\'{e}fense, 200 Avenue de la R\'{e}publique, 92000 Nanterre, and LPMA, Universit\'{e} Paris VI, Paris, France. Email: \texttt{zhengcb@gmail.com }}}

\begin{document}
\maketitle

\begin{abstract}
We study multi-dimensional normal approximations on the Poisson space by means of Malliavin calculus, Stein's method and probabilistic interpolations. Our results yield new multi-dimensional central limit theorems for multiple integrals with respect to Poisson measures -- thus significantly extending previous works by Peccati, Sol\'{e}, Taqqu and Utzet. Several explicit examples (including in particular vectors of linear and non-linear functionals of Ornstein-Uhlenbeck L\'{e}vy processes) are discussed in detail.\\


\end{abstract}
\tableofcontents
\begin{section}{Introduction}
Let $(Z,\mathcal{Z},\mu) $ be a measure space such that
$Z$ is a Borel space and $\mu$ is a $\sigma$-finite non-atomic
Borel measure. We set $\mathcal{Z}_{\mu} = \{ B\in \mathcal{Z}: \mu(B)< \infty \}$. In what follows, we write
$\hat{N} = \{\hat{N}(B) : B\in \mathcal{Z}_{\mu} \} $ to indicate a {\sl compensated Poisson measure} on $(Z,\mathcal{Z}) $ with {\sl control} $\mu$. In other words, $\hat{N} $ is a collection of random variables defined on some probability space $(\Omega, \mathcal{F}, \p) $, indexed by
the elements of $\mathcal{Z}_{\mu} $ and such that: (i) for every $B,C \in \mathcal{Z}_{\mu}$ such that $B \cap C = \varnothing$, the random variables $ \hat{N}(B)$ and $ \hat{N}(C)$ are independent;  (ii) for every $B \in \mathcal{Z}_{\mu} $, $\hat{N}(B) \law N(B)-\mu(B) $, where $N(B) $ is a Poisson random variable with paremeter $\mu(B) $. A random measure verifying property (i) is customarily called ``completely random'' or, equivalently, ``independently scattered'' (see e.g. \cite{surg}). \\

Now fix $d\geq 2$, let $F=(F_1,\ldots,F_d) \subset L^2(\sigma(\hat{N}),\p)$ be a vector of square-integrable functionals of $\hat{N}$, and let $X=(X_1,\ldots,X_d)$ be a centered Gaussian vector. The aim of this paper is to develop several techniques, allowing to assess quantities of the type
\begin{equation}\label{dist}
d_{\mathcal{H}} (F,X) = \sup\limits_{g\in \mathcal{H}} |\E[g(F)]-\E[g(X)] |,
\end{equation}
where $\mathcal{H}$ is a suitable class of real-valued test functions on $\Rd$.  As discussed below, our principal aim is the derivation of explicit upper bounds in multi-dimensional Central limit theorems (CLTs) involving vectors of general functionals of $\hat{N}$. Our techniques rely on a powerful combination of Malliavin calculus (in a form close to Nualart and Vives \cite{nuaviv}), Stein's method for multivariate normal approximations (see e.g. \cite{chameck, npr, rr} and the references therein), as well as some interpolation techniques reminiscent of Talagrand's ``smart path method'' (see \cite{talag}, and also \cite{cha2, npr2}). As such, our findings can be seen as substantial extensions of the results and techniques developed e.g. in \cite{np, npr, pstu}, where Stein's method for normal approximation is successfully combined with infinite-dimensional stochastic analytic procedures (in particular, with infinite-dimensional integration by parts formulae).\\

The main findings of the present paper are the following:\\

{\bf \noindent(I)} We shall use both Stein's method and interpolation procedures in order to obtain explicit upper bounds for distances such as (\ref{dist}). Our bounds will involve Malliavin derivatives and infinite-dimensional Ornstein-Uhlenbeck operators. A careful use of interpolation techniques also allows to consider Gaussian vectors with a non-positive definite covariance matrix. As seen below, our estimates are the exact Poisson counterpart of the bounds deduced in a Gaussian framework in Nourdin, Peccati and R\'{e}veillac \cite{npr} and Nourdin, Peccati and Reinert \cite{npr2}.  \\

{\bf \noindent(II)} The results at point {\bf (I)} are applied in order to derive explicit sufficient conditions for multivariate CLTs involving vectors
of multiple Wiener-It\^o integrals with respect to $\hat{N}$. These results extend to arbitrary orders of integration and arbitrary dimensions the CLTs deduced by Peccati and Taqqu \cite{pt} in the case of single and double Poisson integrals (note that the techniques developed in \cite{pt} are based on decoupling). Moreover, our findings partially generalize to a Poisson framework the main result by Peccati and Tudor \cite{ptudor}, where it is proved that, on a Gaussian Wiener chaos (and under adequate conditions), componentwise convergence to a Gaussian vector is always equivalent to joint convergence. (See also \cite{npr}.) As demonstrated in Section 6, this property is particularly useful for applications.\\

The rest of the paper is organized as follows. In Section 2 we discuss some preliminaries, including basic notions of stochastic analysis on the Poisson space and Stein's method for multi-dimensional normal approximations. In Section 3, we use Malliavin-Stein techniques to deduce explicit upper bounds for the Gaussian approximation of a vector of functionals of a Poisson measure. In Section 4, we use an interpolation method (close to the one developed in \cite{npr2}) to deduce some variants of the inequalities of Section 3. Section 5 is devoted to CLTs for vectors of multiple Wiener-It\^{o} integrals. Section 6 focuses on examples, involving in particular functionals of Ornstein-Uhlenbeck L\'{e}vy processes. An Appendix (Section 7) provides the precise definitions and main properties of the Malliavin operators that are used throughout the paper.

\end{section}

\begin{section}{Preliminaries}
\begin{subsection}{Poisson measures}
As in the previous section, $(Z,\mathcal{Z},\mu) $ is a Borel measure space, and $\hat{N}$ is a Poisson measure on $Z$ with control $\mu$.

\begin{remark} \label{rmk1}
\rm{
Due to the assumptions on the space $(Z,\mathcal{Z},\mu)$,
we can always set $(\Omega,\mathcal{F},\mathbb{P})$ and $\hat{N}$ to be such that
$$ \Omega = \left\{ \omega = \sum_{j=0}^{n} \delta_{z_j},n\in   \mathbb{N} \cup \{\infty\},z_j\in Z    \right\} $$
where $\delta_z$ denotes the Dirac mass at $z$, and $\hat{N}$ is the
\textbf{compensated canonical mapping}
$$\omega\mapsto \hat{N}(B)(\omega) = \omega(B) -\mu(B),\quad B\in \mathcal{Z}_{\mu},\quad \omega\in\Omega,$$
(see e.g. \cite{pic} for more details). For the rest of the paper, we assume that $\Omega$ and $\hat{N}$ have this form. Moreover, the $\sigma$-field $\mathcal{F}$ is supposed to be the $\p$-completion of the $\sigma$-field generated by $\hat{N}$. }
\end{remark}


Throughout the paper, the symbol $L^2(\mu)$ is shorthand for $L^2(Z,\mathcal{Z},\mu)$.
  For $n\geq 2$, we write $L^2(\mu^n)$ and $L^2_s(\mu^n)$, respectively, to indicate the space of real-valued functions on $Z^n$ which are square-integrable with respect to the product measure $\mu^n$, and the subspace of $L^2(\mu^n)$ composed of symmetric functions. Also, we adopt the convention $L^2(\mu) = L_s^2(\mu) =L^2(\mu^1) =L_s^2(\mu^1) $ and use the following standard notation: for every $n\geq 1$ and every $f,g\in L^2(\mu^n)$,
  $$ \langle f,g \rangle_{L^2(\mu^n)} = \int_{Z^n} f(z_1,...,z_n)g(z_1,...,z_n)\mu^n (dz_1,...,dz_n), \quad \|f\|_{L^2(\mu^n)} = \langle f,f \rangle^{1/2}_{L^2(\mu^n)} . $$

%


For every $f\in L^2(\mu^n)$, we denote by $\widetilde{f}$ the canonical symmetrization of $f$, that is,
      $$\widetilde{f}(x_1,\ldots,x_{n})=\cfrac{1}{n!}
\sum_\sigma f (x_{\sigma(1)},\ldots,x_{\sigma(n)}) $$
where $\sigma $ runs over the $n! $ permutations of the
set $\{1,\ldots,n \}$. Note that, e.g. by Jensen's inequality,
\begin{eqnarray} \label{ineq_sym}
\|\tilde{f}\|_{L^2(\mu^n)} \leq  \|f\|_{L^2(\mu^n)}
\end{eqnarray}
For every $f\in L^2_s(\mu^n) $, $n\geq 1$, and every fixed $z\in Z$, we write $f(z,\cdot)$ to indicate the function defined on $Z^{n-1}$ given by $(z_1,\ldots,z_{n-1}) \mapsto f(z,z_1,\ldots,z_{n-1})$. Accordingly, $\widetilde{f(z,\cdot)}$ stands for the symmetrization of the function $f(z,\cdot)$ (in $(n-1)$ variables). Note that, if $n=1$, then $f(z,\cdot)=f(z)$ is a constant.

\begin{definition}
For every deterministic function $h\in L^2(\mu)$, we write
$I_1(h)=\hat{N}(h) = \int_Z h(z) \hat{N}(dz) $ to indicate the {\bf Wiener-Itô
integral} of $h$ with respect to $\hat{N}$. For every $n\geq 2$ and every $f\in L^2(\mu^n)$, we denote by $I_n(f)$
the {\bf multiple Wiener-Itô integral}, of order $n$, of $f$ with respect to $\hat{N}$. We also set $I_n(f)=I_n(\tilde{f})$, for every $f\in L^2(\mu^n)$, and $I_0(C)=C$ for every constant $C$.
\end{definition}

The reader is referred e.g. to Privault \cite{priv} for a complete discussion of multiple Wiener-Itô integrals and their properties (including the forthcoming Proposition \ref{P : MWIone} and Proposition \ref{P: MWIchaos}) -- see also \cite{nuaviv, surg}.

\begin{proposition}\label{P : MWIone}
The following properties hold for every $n,m\geq 1$, every $f\in L_s^2(\mu^n)$  and every $g\in L_s^2(\mu^m)$:
\begin{enumerate}
  \item $\E[I_n(f)]=0$,
  \item $\E[I_n(f) I_m(g)]= n!\langle f,g  \rangle_{L^2(\mu^n)} \1_{(n=m)} $
  (isometric property).
\end{enumerate}
\end{proposition}
The Hilbert space composed of the random variables with the form $I_n(f)$, where $n\geq 1$ and $f\in L^2_s(\mu^n)$, is called the $n$th \emph{Wiener chaos} associated with the Poisson measure $\hat{N}$. The following well-known \emph{chaotic representation property} is essential in this paper.

\begin{proposition}
[Chaotic decomposition] \label{P: MWIchaos} Every random variable $F\in L^2(\mathcal{F},\p)=L^2(\p)$
admits a (unique) chaotic decomposition of the type
\begin{equation}  \label{chao}
F= \E[F] + \sum_{n \geq 1}^{\infty} I_n(f_n)
\end{equation}
where the series converges in $L^2(\p)$ and, for each $n\geq 1$, the kernel $f_n$ is an element
of $L^2_s(\mu^n)$.
\end{proposition}


\end{subsection}

\begin{subsection}{Malliavin operators}
For the rest of the paper, we shall use definitions and results related to Malliavin-type operators defined on the space of functionals of the Poisson measure $\hat{N}$. Our formalism is analogous to the one introduced by Nualart and Vives \cite{nuaviv}. In particular, we shall denote by $D$, $\delta$, $L$ and $L^{-1}$, respectively, the {\sl Malliavin derivative}, the {\sl divergence operator}, the {\sl Ornstein-Uhlenbeck generator} and its {\sl pseudo-inverse}. The domains of $D$, $\delta$ and $L$ are written ${\rm dom} D$, ${\rm dom} \delta$ and ${\rm dom} L$. The domain of $L^{-1}$ is given by the subclass of $L^2(\mathbb{P})$ composed of centered random variables, denoted by $L_0^2(\mathbb{P})$.

Albeit these objects are fairly standard, for the convenience of the reader we have collected some crucial definitions and results in the Appendix (see Section \ref{APPENDIX}). Here, we just recall that, since the underlying probability space $\Omega$ is assumed to be the collection of discrete measures described in Remark \ref{rmk1}, then one can meaningfully define the random variable $\omega\mapsto F_z (\omega) =F(\omega + \delta_z),\, \omega \in \Omega, $  for every given random variable $F$ and every $z\in Z$, where $\delta_z$ is the Dirac mass at $z$. One can therefore prove that the following neat representation of $D$ as a {\sl difference operator} is in order.
\begin{lemma}\label{L : diff}
For each $F\in {\rm dom} D$,
$$ D_z F = F_z - F ,\,\,  \text{a.e.-} \mu(dz). $$
\end{lemma}
A proof of Lemma \ref{L : diff} can be found e.g. in \cite{nuaviv, pstu}. Also, we will often need the forthcoming Lemma \ref{deltaDelle}, whose proof can be found in \cite{pstu} (it is a direct consequence of the definitions of the operators $D$, $\delta$ and $L$).
\begin{lemma}\label{deltaDelle} One has that
$F\in {\rm dom} L$ if and only if $F\in {\rm dom} D$ and $DF \in \rm{dom}\delta$, and in this case
$$ \delta D F = -LF.$$
\end{lemma}

\begin{remark}{\rm
For every $F\in L^2_0(\p)$, it holds that $L^{-1}F \in {\rm dom}L$, and consequently
$$ F=L L^{-1} F = \delta(-D L^{-1} F)=-\delta(D L^{-1} F)  .$$ }
\end{remark}
\end{subsection}

\begin{subsection}{Products of stochastic integrals and star contractions}
In order to give a simple description of the \emph{multiplication formulae} for multiple Poisson integrals (see formula (\ref{product})), we (formally) define a contraction kernel $f \star_r^l g$ on $Z^{p+q-r-l}$ for functions $f\in L^2_s(\mu^p) $ and $g \in L^2_s(\mu^q) $, where $p,q \geq 1$, $r=1,\ldots, p\wedge q$ and $l=1,\ldots,r $, as follows:
\begin{eqnarray}
& & f \star_r^l
g(\gamma_1,\ldots,\gamma_{r-l},t_1,,\ldots,t_{p-r},s_1,,\ldots,s_{q-r}) \label{contraction} \\
&=& \int_{Z^l} \mu^l(dz_1,...,dz_l)
f(z_1,,\ldots,z_l,\gamma_1,\ldots,\gamma_{r-l},t_1,,\ldots,t_{p-r}) \nonumber \\
& & \quad\quad\quad\quad\quad\quad\quad\quad\quad\quad\quad\quad \times g(z_1,,\ldots,z_l,\gamma_1,\ldots,\gamma_{r-l},s_1,,\ldots,s_{q-r}). \nonumber
\end{eqnarray}
In other words, the star operator ``$\,\star_r^l\,$'' reduces the number of variables in the tensor product of $f$ and $g$ from $p+q$ to $p+q-r-l$: this operation is realized by first identifying $r$ variables in $f$ and $g$, and then by integrating out $l$ among them. To deal with the case $l=0$ for $r=0,\ldots, p\wedge q$ , we set
\begin{eqnarray*}
& &f \star_r^0
g(\gamma_1,\ldots,\gamma_{r},t_1,,\ldots,t_{p-r},s_1,,\ldots,s_{q-r}) \\
&=& f(\gamma_1,\ldots,\gamma_{r},t_1,,\ldots,t_{p-r})
g(\gamma_1,\ldots,\gamma_{r},s_1,,\ldots,s_{q-r}),
\end{eqnarray*}
 and
$$ f \star_0^0 g (t_1,,\ldots,t_{p},s_1,,\ldots,s_{q}) =f\otimes g (t_1,,\ldots,t_{p},s_1,,\ldots,s_{q})= f(t_1,,\ldots,t_{p})
g(s_1,,\ldots,s_{q}). $$
By using the Cauchy-Schwarz inequality, one sees immediately that $f \star_r^r g$ is square-integrable for any choice of
$r=0,\ldots, p\wedge q$ , and every $f\in L^2_s(\mu^p) $, $g \in L^2_s(\mu^q) $. \\

As e.g. in \cite[Theorem 4.2]{pstu}, we will sometimes need to work under some specific regularity assumptions for the kernels that are the object of our study.

\begin{definition}\label{DEF : ASSmptns} Let $p\geq 2$ and let $f\in L^2_s(\mu^p)$.
\begin{enumerate}
\item The kernel $f$ is said to satisfy {\bf Assumption A}, if $(f\star_{p}^{p-r}f)\in L^2(\mu^r)$ for every $r=1,...,p$. Note that $(f\star_{p}^{0}f)\in L^2(\mu^p)$ if and only if $f\in L^4(\mu^p)$.
\item The kernel $f$ is said to satisfy {\bf Assumption B}, if every contraction of the type
$$(z_1,...,z_{2p-r-l})\mapsto | f | \star_r^l | f | (z_1,...,z_{2p-r-l})$$ is
\textit{well-defined and finite} for every $r=1,...,p$, every
$l=1,...,r$ and every $(z_1,...,z_{2p-r-l})\in Z^{2p-r-l}$.
\end{enumerate}
\end{definition}

The following statement will be used in order to deduce the multivariate CLT stated in Theorem \ref{CLT}. The proof is left to the reader: it is a consequence of the Cauchy-Schwarz inequality and of the Fubini theorem (in particular, Assumption A is needed in order to implicitly apply a Fubini argument -- see step (S4) in the proof of Theorem 4.2 in \cite{pstu} for an analogous use of this assumption).

\begin{lemma} \label{lemma_ineq}
Fix integers $p,q\geq 1 $, as well as kernels
$f\in L^2_s(\mu^p)$ and $ g \in L^2_s(\mu^q)$ satisfying Assumption A in Definition \ref{DEF : ASSmptns}. Then, for any integers $s,t$ satisfying
$1 \leq s \leq t \leq p\wedge q$, one has that $f\star_t^s g \in L^2(\mu^{p+q-t-s})$, and moreover
\begin{enumerate}
  \item
$$\|f\star^s_t g\|^2_{L^2(\mu^{p+q-t-s})} = \langle f\star^{p-t}_{p-s}f , g\star^{q-t}_{q-s}g \rangle_{L^2(\mu^{t+s})},$$ (and, in particular,
$$ \|f\star^s_t f\|_{L^2(\mu^{2p-s-t})} = \| f\star^{p-t}_{p-s}f \|_{L^2(\mu^{t+s})}\, );$$
  \item
  \begin{eqnarray*}
\|f\star^s_t g\|^2_{L^2(\mu^{p+q-t-s})} &\leq&  \|f\star^{p-t}_{p-s}f\|_{L^2(\mu^{t+s})} \times \|g\star^{q-t}_{q-s}g\|_{L^2(\mu^{t+s})}  \\
 &=& \|f\star^s_t f\|_{L^2(\mu^{2p-s-t})} \times \|g\star^s_t g\|_{L^2(\mu^{2q-s-t})} .
\end{eqnarray*}
\end{enumerate}
\end{lemma}

\begin{remark}\label{REMK : ASSUMPTIONS}{\rm

\begin{enumerate}
\item Writing $k=p+q-t-s$, the requirement that $1 \leq s \leq t \leq p\wedge q$ implies that $|q-p| \leq k \leq p+q-2 $.

\item One should also note that, for every
$1\leq p\leq q$ and every
$r=1,...,p$,
\begin{equation}\label{useful}
\int_{Z^{p+q-r}} (f\star_r^0 g)^2 d\mu^{p+q-r} = \int_{Z^r}
(f\star_p^{p-r} f)(g\star_q^{q-r} g) d\mu^r,
\end{equation}
for every $f\in L_s^2(\mu^p)$ and every $g\in L_s^2(\mu^q)$, not necessarily verifying Assumption A. Observe that the integral on the RHS of (\ref{useful}) is well-defined, since $f\star_p^{p-r} f\geq 0$ and $g\star_q^{q-r} g \geq 0$.
\item Fix $p,q\geq 1$, and assume again that $f\in L^2_s(\mu^p)$ and $ g \in L^2_s(\mu^q)$ satisfy Assumption A in Definition \ref{DEF : ASSmptns}. Then, a consequence of Lemma \ref{lemma_ineq} is that, for every $r=0,...,p\wedge q-1$ and every $l=0,...,r$, the kernel $f(z,\cdot)\star_r^l g(z,\cdot)$ is an element of $L^2(\mu^{p+q-t-s-2})$ for $\mu(dz)$-almost every $z\in Z$.
\end{enumerate}
}
\end{remark}

To conclude the section, we present an important {\it product formula} for Poisson multiple integrals (see e.g. \cite{kab, surg2} for a proof).
\begin{proposition}
[Product formula] Let $f\in L^2_s(\mu^p) $ and $g\in
L^2_s(\mu^q)$, $p,q\geq 1 $, and suppose moreover that $f \star_r^l g
\in L^2(\mu^{p+q-r-l})$ for every $r=1,\ldots,p\wedge q $ and $
l=1,\dots,r$ such that $l\neq r $. Then,
\begin{equation} \label{product}
I_p(f)I_q(g) = \sum_{r=0}^{p\wedge q} r!
\left(
\begin{array}{c}
  p\\
  q\\
\end{array}
\right)
 \left(
\begin{array}{c}
  q\\
  r\\
\end{array}
\right)
 \sum_{l=0}^r
 \left(
\begin{array}{c}
  r\\
  l\\
\end{array}
\right)  I_{p+q-r-l} \left(\widetilde{f\star_r^l g}\right),
\end{equation}
 with the tilde $\sim$ indicating a symmetrization, that
is,
$$\widetilde{f\star_r^l g}(x_1,\ldots,x_{p+q-r-l})=\cfrac{1}{(p+q-r-l)!}
\sum_\sigma f\star_r^l g(x_{\sigma(1)},\ldots,x_{\sigma(p+q-r-l)}), $$
where $\sigma $ runs over all $(p+q-r-l)! $ permutations of the
set $\{1,\ldots,p+q-r-l \}$.

\end{proposition}
\end{subsection}

\begin{subsection}{Stein's method: measuring the distance between random vectors}

We write $g\in \C^k(\Rd)$ if the function $g:\Rd \rightarrow\R$ admits continuous partial derivatives up to the order $k$.

\begin{definition}
\begin{enumerate}
  \item The {\bf Hilbert-Schmidt inner product} and the
  {\bf Hilbert - Schmidt norm} on the class of $d\times d$ real
  matrices, denoted respectively by $\langle \cdot,\cdot
  \rangle_{H.S.}$ and $\|\cdot \|_{H.S.}$, are defined as follows:
  for every pair of matrices $A$ and $B$, $\langle A,B \rangle_{H.S.} := Tr(AB^T) $
  and $\|A \|_{H.S.}=\sqrt{\langle A,A \rangle_{H.S.}} $, where $Tr(\cdot)$ indicates the usual trace operator.
  \item The \textbf{operator norm} of a $d\times d$ real matrix $A$
  is given by $\|A \|_{op} := \sup_{\|x\|_{\Rd}=1} \|Ax\|_{\Rd}$.
 \item  For every function $g:\Rd \mapsto \R$, let
  $$\|g\|_{Lip}:=\sup\limits_{x\neq y} \cfrac{|g(x)-g(y)|}{\|x-y\|_{\Rd}}, $$
  where $\|\cdot \|_{\Rd}$ is the usual Euclidian norm on
  $\Rd$.
  If $g\in \C^1(\Rd)$, we also write
  $$M_2(g):= \sup\limits_{x\neq y} \cfrac{\|\nabla g(x)-\nabla g(y)\|_{\Rd}}{\|x-y\|_{\Rd}}, $$
  If $g\in \C^2(\Rd)$,
  $$M_3(g):= \sup\limits_{x\neq y} \cfrac{\|\Hess g(x)-\Hess g(y)\|_{op}}{\|x-y\|_{\Rd}}, $$
where $\Hess g(z)$ stands for the Hessian matrix of $g$ evaluated at a point $z$.
\item  For a positive integer $k$ and a function $g\in \C^k(\Rd)$ , we set
$$  \|g^{(k)}\|_{\infty}=\max\limits_{1\leq i_1 \leq \ldots \leq i_k\leq d} \sup\limits_{x\in \Rd} \left| \cfrac{\partial^k}{\partial x_{i_1} \ldots \partial x_{i_k}} g(x)\right| .$$
In particular, by specializing this definition to $g^{(2)} = g''$ and $g^{(3)} = g'''$, we obtain
$$  \|g''\|_{\infty}=\max\limits_{1\leq i_1 \leq i_2 \leq d} \sup\limits_{x\in \Rd} \left| \cfrac{\partial^2}{\partial x_{i_1}  \partial x_{i_2}} g(x)\right| .$$
$$  \|g'''\|_{\infty}=\max\limits_{1\leq i_1 \leq i_2  \leq i_3\leq d} \sup\limits_{x\in \Rd} \left| \cfrac{\partial^3}{\partial x_{i_1} \partial x_{i_2} \partial x_{i_3}} g(x)\right| .$$
\end{enumerate}
\end{definition}
\begin{remark}
\begin{enumerate}{\rm
  \item The norm $\|g\|_{Lip}$ is written $M_1(g)$ in \cite{chameck}.
  \item If $g\in \C^1(\Rd)$, then $\|g\|_{Lip} = \sup\limits_{x\in \Rd} \|\nabla g(x)\|_{\Rd} $.  If $g\in \C^2(\Rd)$, then $$M_2(g) = \sup_{x\in \Rd} \|\Hess g(x)\|_{op} .$$
}
\end{enumerate}

\end{remark}

\begin{definition}
 The  distance $d_2$ between the laws of
  two $\Rd$-valued random vectors $X$ and $Y$ such that $\mathbb{E}\|X\|_{\R^d},\,\mathbb{E}\|Y\|_{\R^d}<\infty$, written $d_{2}(X,Y)$,
  is given by
  $$ d_{2}(X,Y) = \sup_{g\in \mathcal{H}}|\E[g(X)]-\E[g(Y)]|, $$
  where $\mathcal{H}$ indicates the collection of all functions $g\in \C^2(\Rd)$
  such that $\|g\|_{Lip}\leq 1$ and $M_2(g)\leq 1$.
\end{definition}
\begin{definition}
 The  distance $d_3$ between the laws of
  two $\Rd$-valued random vectors $X$ and $Y$  such that $\mathbb{E}\|X\|^2_{\R^d},\,\mathbb{E}\|Y\|^2_{\R^d}<\infty$, written $d_{3}(X,Y)$,
  is given by
  $$ d_{3}(X,Y) = \sup_{g\in \mathcal{H}}|\E[g(X)]-\E[g(Y)]|, $$
  where $\mathcal{H}$ indicates the collection of all functions $g\in \C^3(\Rd)$
  such that $\|g''\|_{\infty} \leq 1$ and $\|g'''\|_{\infty} \leq 1$.
\end{definition}

\begin{remark}\label{RM : CVinLAW}{\rm The distances $d_2$ and $d_3$ are related, respectively, to the estimates of Section \ref{Sect : MALLIAVINb} and Section \ref{SEC : INTERPLT}. Let $j=2,3$. It is easily seen that, if $d_j(F_n,F)\rightarrow 0$, where $F_n,F$ are random vectors in $\Rd$, then necessarily $F_n$ converges in distribution to $F$. It will also become clear later on that, in the definition of $d_2$ and $d_3$, the choice of the constant 1 as a bound for $\|g\|_{Lip},\, M_2(g),\, \|g''\|_{\infty},\,\|g'''\|_{\infty}$ is arbitrary and immaterial for the derivation of our main results (indeed, we defined $d_2$ and $d_3$ in order to obtain bounds as simple as possible). See the two tables in Section \ref{SS : tables} for a list of available bounds involving more general test functions.

}
\end{remark}

The following result is a $d$-dimensional version of Stein's Lemma; analogous statements can be found in \cite{chameck, npr, rr} -- see also Barbour \cite{BARB} and G\"otze \cite{GTZ}, in connection with the so-called ``generator approach'' to Stein's method. As anticipated, Stein's Lemma will be used to deduce an explicit bound on the distance $d_2$ between the law of a vector of functionals of $\hat{N}$ and the law of a Gaussian vector. To this end, we need the two estimates (\ref{ineq_lip}) (which is proved in \cite{npr}) and (\ref{ineq_M2}) (which is new). \\

From now on, given a $d\times d$ nonnegative definite matrix $C$, we write $\mathcal{N}_d (0,C)$ to indicate the law of a centered $d$-dimensional Gaussian vector with covariance $C$.

\begin{lemma}[Stein's Lemma and estimates] \label{Stein_multidim}
Fix an integer $d\geq 2$ and let $C=\{C(i,j): i,j=1,\ldots,d \} $
be a $d\times d $ nonnegative definite symmetric real matrix.
\begin{enumerate}
    \item Let $Y$ be a random variable with values in $\Rd$. Then $Y\sim \mathcal{N}_d (0,C) $
    if and only if, for every twice differentiable function $f:\Rd \mapsto \R $ such that
    $\E|\langle C, \Hess f(Y) \rangle_{H.S.}|+ \E|\langle Y,\nabla f(Y )\rangle_{\Rd}|<\infty $, it holds that
    $$ \E[\langle Y,\nabla f(Y )\rangle_{\Rd} - \langle C, \Hess f(Y) \rangle_{H.S.}] = 0  $$
    \item Assume in addition that $C$ is positive definite and consider a Gaussian random vector $X \sim \mathcal{N}_d (0,C)
    $. Let $g:\Rd \mapsto \R $ belong to $\C^2 (\Rd )
    $ with first and second bounded derivatives. Then, the function $U_0
    (g)$ defined by
    $$U_0 g(x) := \int_0^1 \cfrac{1}{2t} \E[g(\sqrt{t}x+ \sqrt{1-t}X )-g(X)]dt $$
    is a solution to the following partial differential equation (with unknown function
    $f$):
    $$ g(x)-\E[g(X)] = \langle x,\nabla f(x)\rangle_{\Rd} - \langle C, \Hess f(x) \rangle_{H.S.}, \,\, x\in \Rd. $$
    Moreover, one has that
    \begin{equation} \label{ineq_lip}
    \sup_{x\in \Rd} \|\Hess U_0 g(x) \|_{H.S.} \leq
    \|C^{-1}\|_{op}\, \|C\|_{op}^{1/2} \|g\|_{Lip},
    \end{equation}
    and
    \begin{equation}  \label{ineq_M2}
    M_3(U_0 g) \leq\frac{\sqrt{2\pi}}{4} \|C^{-1}\|^{3/2}_{op}\, \|C\|_{op}\, M_2(g).
    \end{equation}
\end{enumerate}
\end{lemma}
\begin{proof}
We shall only show relation (\ref{ineq_M2}), as the proof of the remaining points in the statement
can be found in \cite{npr}. Since $C$ is a positive definite matrix, there exists a
non-singular symmetric matrix $A$ such that $A^2=C$, and $A^{-1}X \sim \Nd(0,I_d) $. Let $U_0 g(x)=h(A^{-1}x) $,
where
$$ h(x)=\int_0^1 \cfrac{1}{2t} \E[g_A (\sqrt{t}x+\sqrt{1-t}A^{-1}X) - g_A (A^{-1}X) ]dt $$
and $g_A(x)=g(Ax)$. As $A^{-1}X \sim \Nd(0,I_d) $, the function
$h$ solves the Stein's equation
$$ \langle x,\nabla h(x) \rangle_{\Rd} -\Delta h(x) = g_A(x) -\E[g_A(Y)], $$
where $Y \sim \Nd(0,I_d)$ and $\Delta$ is the Laplacian. On the one hand, as $\Hess g_A(x) =A\,  \Hess g(Ax) A$ (recall that $A$ is symmetric),
we have
\begin{eqnarray*}
 M_2(g_A) &=& \sup\limits_{x\in \Rd} \|\Hess g_A(x)\|_{op}= \sup\limits_{x\in \Rd} \|A \Hess g(Ax)A\|_{op}\\
            &=& \sup\limits_{x\in \Rd} \|A \Hess g(x)A\|_{op} \leq \|A\|_{op}^2 M_2(g)\\
            &=& \|C\|_{op} M_2(g),
\end{eqnarray*}
where the inequality above follows from the well-known relation
$\|AB\|_{op} \leq \|A\|_{op} \|B\|_{op} $. Now write $h_{A^{-1}}(x) = h(A^{-1}x)$: it is
easily seen that
$$\Hess U_0 g(x) = \Hess  h_{A^{-1}}(x) = A^{-1} \Hess  h(A^{-1}x) A^{-1} .$$
It follows that
\begin{eqnarray*}
 M_3(U_0 g) &= &M_3 (h_{A^{-1}}) \\
              &=& \sup\limits_{x\neq y} \cfrac{\|\Hess  h_{A^{-1}}(x)
              - \Hess  h_{A^{-1}}(y)\|_{op}}{\|x-y\|} \\
              &=& \sup\limits_{x\neq y} \cfrac{\|A^{-1} \Hess  h(A^{-1}x) A^{-1}
              - A^{-1} \Hess  h(A^{-1}y) A^{-1}\|_{op}}{\|x-y\|} \\
              &\leq&  \|A^{-1}\|_{op}^2 \times \sup\limits_{x\neq y} \cfrac{\|\Hess
              h(A^{-1}x) - \Hess h(A^{-1}y)\|_{op}}{\|x-y\|}
              \times \cfrac{\|A^{-1}x - A^{-1}y\|}{\|A^{-1}x -A^{-1}y\|} \\
              &\leq&  \|A^{-1}\|_{op}^2 \times \sup\limits_{x\neq y} \cfrac{\|\Hess
              h(A^{-1}x) - \Hess h(A^{-1}y)\|_{op}}{\|A^{-1}x -A^{-1}y\|}
              \times \|A^{-1}\|_{op} \\
              &=& \|C^{-1}\|^{3/2}_{op} M_3(h).
\end{eqnarray*}
Since $M_3(h)\leq\frac{\sqrt{2\pi}}{4} M_2(g_A)$ (according to
\cite[Lemma 3]{chameck}), relation (\ref{ineq_M2}) follows immediately.
\end{proof}

%

\end{subsection}
\end{section}

\begin{section}{Upper bounds obtained by Malliavin-Stein methods}\label{Sect : MALLIAVINb}
We will now deduce one of the main findings of the present paper, namely Theorem \ref{ineq_2W}. This result allows to estimate the distance between the law of a vector of Poisson functionals and the law of a Gaussian vector, by combining the multi-dimensional Stein's Lemma \ref{Stein_multidim} with the algebra of the Malliavin operators. Note that, in this section, all Gaussian vectors are supposed to have a positive definite covariance matrix.


We start by proving a technical lemma, which is a crucial element in most of our proofs.
\begin{lemma} \label{lem1}
Fix $d\geq 1$ and consider a vector of random variables $F:=(F_1,\ldots,F_d)\subset L^2({\p})$. Assume that, for all
$1\leq i\leq d$, $F_i\in
{\rm dom}\, D$, and $\E[F_i]=0$. For all $\phi\in \C^2(\Rd)$ with bounded derivatives, one has that
$$D_z \phi (F_1,\ldots,F_d)=\smd \pxi \phi(F)(D_z F_i) +
\sum\limits_{i,j=1}^{d} R_{ij} (D_z F_i,D_z F_j), \,\, z\in Z, $$
 where the mappings
$R_{ij}$ satisfy
\begin{equation} \label{residue}
|R_{ij}(y_1,y_2)| \leq \cfrac{1}{2} \sup\limits_{x\in \Rd} \big|\cfrac{\partial^2}{\partial x_i \partial x_j} \phi(x)\big| \times |y_1 y_2|
\leq \cfrac{1}{2} \|\phi''\|_{\infty}|y_1 y_2| .
\end{equation}
\end{lemma}

\begin{proof}
By the multivariate Taylor theorem and Lemma \ref{L : diff},
\begin{eqnarray*}
D_z \phi (F_1,\ldots,F_d)
&=& \phi (F_1,\ldots,F_d)(\omega + \delta_z) - \phi (F_1,\ldots,F_d)(\omega) \\
&=& \phi (F_1(\omega + \delta_z),\ldots,F_d(\omega + \delta_z)) -
\phi (F_1(\omega),\ldots,F_d(\omega)) \\
&=& \smd \pxi \phi(F_1(\omega),\ldots,F_d(\omega)) (F_i(\omega + \delta_z) - F_i(\omega))
+ R\\
&=& \smd \pxi \phi(D_z F_i) + R,
\end{eqnarray*}
where the term $R$ represents the residue:
 $$R =R(D_z F_1,\ldots,D_z
F_d)= \sum\limits_{i,j=1}^{d} R_{ij} (D_z F_i,D_z F_j), $$
and the mapping $(y_1,y_2) \mapsto R_{ij}(y_1,y_2)$ verifies
 (\ref{residue}).
\end{proof}
\begin{remark}{\rm
Lemma \ref{lem1} is the Poisson counterpart of the multi-dimensional ``chain rule'' verified by the Malliavin derivative on a Gaussian space (see \cite{np,npr}). Notice that the term $R$ does not appear in the Gaussian framework.
}
\end{remark}

The following result uses the two Lemmas \ref{Stein_multidim} and \ref{lem1}, in order to compute explicit bounds on the distance between the laws of a vector of Poisson functionals and the law of a Gaussian vector.

\begin{theorem}[Malliavin-Stein inequalities on the Poisson space] \label{ineq_2W}
Fix $ d\geq 2 $ and let $ C=\{C(i,j): i,j= 1,\ldots,d  \} $ be a
$d\times d $ positive definite matrix. Suppose that $ X\sim
\Nd(0,C) $ and that $ F=(F_1,\ldots, F_d ) $ is a $ \Rd $-valued
random vector such that $ \E[F_i]=0 $ and $ F_i \in {\rm dom}\, D$, $i=1,\ldots,d$. Then,
\begin{eqnarray}
d_{2}(F,X) \leq \|C^{-1}\|_{op} \|C\|_{op}^{1/2} \sqrt{\sum_{i,j=1}^{d} \E[(C(i,j) - \langle  DF_i,-DL^{-1}F_j \rangle_{L^2(\mu)} )^2 ] }  \label{L1}\\
+ \cfrac{\sqrt{2\pi}}{8} \|C^{-1}\|^{3/2}_{op} \|C\|_{op} \int_Z \mu(dz)\E\left[\left(\smd|D_z F_i | \right)^2 \left(\smd|D_z L^{-1} F_i |  \right) \right]
 .\label{L2}
\end{eqnarray}
\end{theorem}
\begin{proof}
If either one of the expectations in (\ref{L1}) and (\ref{L2}) are infinite, there is nothing to prove: we shall therefore work under the assumption that both expressions (\ref{L1})--(\ref{L2}) are finite. By the definition of the distance $d_2$, and by using an interpolation argument (identical to the one used at the beginning of the proof of Theorem 4 in \cite{chameck}), we need only show the following inequality:
\begin{eqnarray}  \nonumber
&& |\E[g(X)]-\E[g(F)]| \\
&& \leq A \|C^{-1}\|_{op} \|C\|_{op}^{1/2} \sqrt{\sum_{i,j=1}^{d} \E[(C(i,j) - \langle  DF_i,-DL^{-1}F_j \rangle_{L^2(\mu)} )^2 ] } \label{eq1} \\
&&+ \cfrac{\sqrt{2\pi}}{8}B\|C^{-1}\|^{3/2}_{op} \|C\|_{op} \int_Z
\mu(dz)\E\left[\left(\smd|D_z F_i | \right)^2 \left(\smd|D_z L^{-1} F_i |  \right) \right] \nonumber
\end{eqnarray}
for any $ g\in \C^\infty(\Rd)$ with first and second bounded derivatives, such that  $\|g\|_{Lip} \leq A $ and $ M_2(g)\leq B  $.
To prove (\ref{eq1}), we use Point (ii) in Lemma
\ref{Stein_multidim} to deduce that

\begin{eqnarray*}
&&|\E[g(X)]-\E[g(F)]| \\
&&= |\E[\langle C, \Hess U_0 g(F)\rangle_{H.S.} - \langle F, \nabla U_0 g(F) \rangle_{\Rd} ]| \\
&&= \left|\E\left[\sum_{i,j=1}^{d} C(i,j) \pxij U_0 g(F) - \sum_{k=1}^{d} F_k \pxk U_0 g(F)\right]\right|\\
&&= \left|\sum_{i,j=1}^{d} \E\left[ C(i,j) \pxij U_0 g(F)\right] + \sum_{k=1}^{d} \E\left[\delta(DL^{-1}F_k) \pxk U_0 g(F)\right]\right|\\
&&= \left|\sum_{i,j=1}^{d} \E\left[ C(i,j) \pxij U_0 g(F)\right] - \sum_{k=1}^{d}
\E\left[\left\langle D\left(\pxk U_0 g(F) \right), -DL^{-1}F_k\right\rangle_{L^2(\mu)}\right]  \right|.
\end{eqnarray*}
We write $ \pxk U_0 g(F):= \phi_k (F_1,\ldots,F_d) =
\phi_k (F) $. By using Lemma \ref{lem1}, we infer
\begin{eqnarray*}
D_z \phi_k (F_1,\ldots,F_d)
 &=& \smd \pxi \phi_k(F)(D_z F_i) + R_{k},
\end{eqnarray*}
with $R_{k} = \sum\limits_{i,j=1}^{d} R_{i,j,k} (D_z F_i,D_z F_j)
$, and
$$|R_{i,j,k}(y_1,y_2)| \leq \cfrac{1}{2} \sup\limits_{x\in \Rd} \left|\pxij \phi_k (x)\right|\times |y_1 y_2|. $$
It follows that
\begin{eqnarray*}
& & |\E[g(X)]-\E[g(F)]|\\
&=& \left|\sum_{i,j=1}^{d} \E\left[ C(i,j) \pxij U_0 g(F)\right] -
\sum_{i,k=1}^{d} \E\left[\cfrac{\partial^2}{\partial x_i \partial x_k} (U_0 g(F))\langle DF_i,
-DL^{-1}F_k\rangle_{L^2(\mu)} \right] \right.\\
& & \left. + \sum\limits_{i,j,k=1}^{d}  \E\left[\langle R_{i,j,k}(DF_i,DF_j),-DL^{-1}F_k\rangle_{L^2(\mu)}\right] \right| \\
&\leq& \sqrt{\E[\|\Hess U_0 g(F)\|^2_{H.S.}]} \times \sqrt{
\sum_{i,j=1}^d\E\left[\left(C(i,j)- \langle DF_i, -DL^{-1}F_j\rangle_{L^2(\mu)}\right)^2\right]} + |R_2 |,
\end{eqnarray*}
where
$$R_2 = \sum\limits_{i,j,k=1}^{d} \E[\langle R_{i,j,k}(DF_i,DF_j),-DL^{-1}F_k\rangle_{L^2(\mu)}] .$$
Note that (\ref{ineq_lip}) implies that $\|\Hess U_0 g(F) \|_{H.S.} \leq
    \|C^{-1}\|_{op}\, \|C\|_{op}^{1/2} \|g\|_{Lip}$.
By using (\ref{ineq_M2}) and the fact $\|g'''\|_{\infty} \leq M_3(g) $, we have
\begin{eqnarray*}
& & |R_{i,j,k} (y_1,y_2)| \leq \cfrac{1}{2} \sup\limits_{x\in \Rd}
\left|\cfrac{\partial^3}{\partial x_i \partial
x_j \partial x_k} U_0(g(y))\right| \times |y_1 y_2| \\
&\leq&
\cfrac{\sqrt{2\pi}}{8} M_2(g) \|C^{-1}\|^{3/2}_{op} \|C\|_{op}
\times |y_1 y_2| \leq
 \cfrac{\sqrt{2\pi}}{8} B\|C^{-1}\|^{3/2}_{op} \|C\|_{op} \times |y_1 y_2|,
\end{eqnarray*}
from which we deduce the desired conclusion.
\end{proof}
\end{section}

\medskip

Now recall that, for a random variable $F=\hat{N}(h)=I_1(h)$ in the first Wiener chaos of $\hat{N}$, one has that $DF=h$ and $L^{-1}F = -F$. By virtue of Remark \ref{RM : CVinLAW}, we immediately deduce the following consequence of Theorem \ref{ineq_2W}.

\begin{corollary}\label{Cor1} For a fixed $d\geq 2$, let $X\sim\mathcal{N}_d(0,C)$, with $C$ positive definite, and let $$F_n=(F_{n,1},...,F_{n,d}) = (\hat{N}(h_{n,1}),...,\hat{N}(h_{n,d})),\,\,n\geq 1,$$ be a collection of $d$-dimensional random vectors living in the first Wiener chaos of $\hat{N}$. Call $K_n$ the covariance matrix of $F_n$, that is: $K_n(i,j) = \mathbb{E}[\hat{N}(h_{n,i})\hat{N}(h_{n,j})] =\langle h_{n,i}, h_{n,j}\rangle_{L^2(\mu)}$. Then,
$$ d_2(F_n,X) \leq \|C^{-1}\|_{op} \|C\|_{op}^{1/2}\, \|C-K_n\|_{H.S.} + \frac{d^2\sqrt{2\pi}}{8}\|C^{-1}\|^{3/2}_{op} \|C\|_{op}\sum_{i=1}^d\int_Z |h_{n,i}(z)|^3\mu(dz). $$
In particular, if
\begin{equation}\label{Panocha}
 K_n(i,j) \rightarrow C(i,j) \,\,\,\, \mbox{and} \,\,\,\, \int_Z |h_{n,i}(z)|^3\mu(dz)\rightarrow 0
\end{equation}
(as $n\rightarrow\infty$ and for every $i,j=1,...,d$), then $d_2(F_n,X) \rightarrow 0$ and $F_n$ converges in distribution to $X$.
\end{corollary}

\begin{remark} \label{rmk2}
{\rm
\begin{enumerate}
\item The conclusion of Corollary \ref{Cor1} is by no means trivial. Indeed, apart from the requirement on the asymptotic behavior of covariances, the statement of Corollary \ref{Cor1} does not contain {\it any} assumption on the joint distribution of the components of the random vectors $F_n$. We will see in Section 5 that analogous results can be deduced for vectors of multiple integrals of arbitrary orders. We will also see in Corollary \ref{Cor2} that one can relax the assumption that $C$ is positive definite.
\item The inequality appearing in the statement of Corollary \ref{Cor1} should also be compared with the following result, proved in \cite{npr}, yielding a bound on the Wasserstein distance between the laws of two Gaussian vectors of dimension $d\geq 2$. Let $Y\sim \mathcal{N}_d(0, K)$ and $X\sim \mathcal{N}_d(0, C)$,
where $K$ and $C$ are two positive definite covariance matrices.
Then,
$
d_{W}(Y,X)\leq  Q(C,K)\times \|C-K\|_{H.S.},
$
where
$$
Q(C,K) :=\min\{\| C^{-1}\|_{op} \,\, \| C\|_{op}^{1/2}, \|
K^{-1}\|_{op} \,\, \| K\|_{op}^{1/2}\},
$$
and $d_{W}$ denotes the Wasserstein distance between the laws of random variables with values in $\Rd$.
\end{enumerate}
}
\end{remark}

\begin{section}{Upper bounds obtained by interpolation methods}\label{SEC : INTERPLT}
\subsection{Main estimates}
In this section, we deduce an alternate upper bound (similar to the ones proved in the previous section) by adopting an approach based on interpolations. We first prove a result involving Malliavin operators.

\begin{lemma} \label{lem2}
Fix $d\geq 1$. Consider $d+1$ random variables $F_i\in L^2(\p)$, $0\leq
i\leq d$, such that $F_i\in {\rm dom}\, D$ and $\E[F_i]=0$. For all $g\in
\C^2(\Rd)$ with bounded derivatives,
$$\E[g(F_1,\ldots,F_d)F_0] \!=\! \E\!\left[\smd \pxi g(F_1,\ldots,F_d)\langle DF_i,-DL^{-1}F_0\rangle_{L^2(\mu)}\right]\! +\! \E\left[\langle R,-DL^{-1}F_0\rangle_{L^2(\mu)}\right]\!,
 $$
 where
\begin{eqnarray}
&& |\E[\langle R,-DL^{-1}F_0\rangle_{L^2(\mu)}]| \label{domus}\\
&&\leq \cfrac{1}{2} \max_{i,j} \sup\limits_{x\in \Rd}
\left| \pxij g(x)\right| \times \int_Z
 \mu(dz) \E\left[\left(\sum_{k=1}^d|D_z F_k|\right)^2 |D_z L^{-1} F_0|\right]. \notag
 \end{eqnarray}
\end{lemma}
\begin{proof}By applying Lemma \ref{lem1},
 \begin{eqnarray*}
 & & \E[g(F_1,\ldots,F_d)F_0] \\
 &=& \E[(LL^{-1}F_0)g(F_1,\ldots,F_d)] \\
 &=& -\E[\delta(DL^{-1}F_0)g(F_1,\ldots,F_d)] \\
 &=& \E[\langle Dg(F_1,\ldots,F_d),-DL^{-1}F_0\rangle_{L^2(\mu)}] \\
 &=& \E\left[\smd \pxi g(F_1,\ldots,F_d)\langle DF_i,-DL^{-1}F_0\rangle_{L^2(\mu)}\right] + \E[\langle R,-DL^{-1}F_0\rangle_{L^2(\mu)}],
 \end{eqnarray*}
and $\E[\langle R,-DL^{-1}F_0\rangle_{L^2(\mu)}]$ verifies the inequality (\ref{domus}).
\end{proof}

\medskip

As anticipated, we will now use an interpolation technique inspired by the so-called ``smart path method'', which is sometimes used in the framework of approximation results for spin glasses (see \cite{talag}). Note that the computations developed below are very close to the ones used in the proof of Theorem 7.2 in \cite{npr2}.

\begin{theorem} \label{ineq_tala}
Fix $ d\geq 1 $ and let $ C=\{C(i,j): i,j= 1,\ldots,d  \} $ be a
$d\times d $ covariance matrix (not necessarily positive definite). Suppose that $ X=(X_1,...,X_d)\sim
\Nd(0,C) $ and that $ F=(F_1,\ldots, F_d ) $ is a $ \Rd $-valued
random vector such that $ \E[F_i]=0 $ and $ F_i \in {\rm dom}\, D$, $i=1,\ldots,d$. Then,
\begin{eqnarray} \label{ineq_tala2}
d_{3}(F,X) &\leq& \cfrac{1}{2} \sqrt{\sum_{i,j=1}^{d} \E[(C(i,j) - \langle  DF_i,-DL^{-1}F_j \rangle_{L^2(\mu)} )^2 ] } \label{unotrino}
 \\
& & + \cfrac{1}{4} \int_Z
 \mu(dz) \E\left[\left(\smd|D_z F_i|\right)^2 \left(\smd |D_z L^{-1} F_i|\right)\right] .\label{unotrino2}
\end{eqnarray}
\end{theorem}
\begin{proof}
We will work under the assumption that both expectations in (\ref{unotrino}) and (\ref{unotrino2}) are finite. By the definition of distance $d_3$, we need only to show the following inequality:
\begin{eqnarray*}
|\E[\phi(X)]-\E[\phi(F)]| &\leq& \cfrac{1}{2} \|\phi'' \|_{\infty}\sum\limits_{i,j=1}^{d} \E[ |C(i,j) - \langle DF_i, -DL^{-1}F_j \rangle_{L^2(\mu)}| ] \\
& & + \cfrac{1}{4}\|\phi'''\|_{\infty} \int_Z
 \mu(dz) \E\left[\left(\smd|D_z F_i|\right)^2 \left(\smd |D_z L^{-1} F_i|\right)\right]
\end{eqnarray*}
for any $ \phi\in \C^3(\Rd)$ with second and third bounded derivatives. Without loss of generality, we may assume that $F$ and $X$ are independent. For $t \in [0,1] $, we set
$$\Psi(t) = \E[ \phi(\sqrt{1-t}(F_1,\ldots, F_d   ) + \sqrt{t} X )]$$
We have immediately
$$|\Psi(1)-\Psi(0) | \leq \sup_{t\in (0,1)} |\Psi'(t)| .$$
Indeed, due to the assumptions on $\phi$, the function $t\mapsto \Psi(t) $ is differentiable on $(0,1)$, and one has also
\begin{eqnarray*}
 \Psi '(t)&=& \sum\limits_{i=1}^{d} \E\left[\pxi \phi\left( \sqrt{1-t}(F_1,\ldots, F_d)
 + \sqrt{t} X \right ) \left( \cfrac{1}{2\sqrt{t}}X_i - \cfrac{1}{2\sqrt{1-t}} F_i \right)\right] \\
 &:=&\frac{1}{2\sqrt{t}} \,\mathfrak{A} - \frac{1}{2\sqrt{1-t}} \,\mathfrak{B}.
\end{eqnarray*}
On the one hand, we have
\begin{eqnarray*}
\mathfrak{A} &=& \sum_{i=1}^d\E\left[\pxi \phi( \sqrt{1-t}(F_1,\ldots, F_d)
 + \sqrt{t} X  ) X_i \right] \\
 &=&\sum_{i=1}^d \E\left[\E\left[\pxi \phi( \sqrt{1-t}a + \sqrt{t} X  ) X_i\right]_{|a=(F_1,
 \ldots, F_d)}  \right] \\
 &=& \sqrt{t} \sum\limits_{i,j=1}^{d} C(i,j)
 \E\left[\E\left[\pxij \phi( \sqrt{1-t}a + \sqrt{t} X  )\right]_{|a=(F_1,
 \ldots, F_d)}  \right] \\  
 &=& \sqrt{t} \sum\limits_{i,j=1}^{d} C(i,j)
 \E\left[\pxij \phi( \sqrt{1-t}(F_1,\ldots, F_d)
   + \sqrt{t} X  ) \right].
\end{eqnarray*}
On the other hand,
\begin{eqnarray*}
\mathfrak{B} &=&\sum_{i=1}^d \E\left[\pxi \phi( \sqrt{1-t}(F_1,\ldots, F_d)
 + \sqrt{t} X  ) F_i \right] \\
&=&\sum_{i=1}^d \E\left[ \E \left[\pxi \phi( \sqrt{1-t}(F_1,\ldots, F_d) + \sqrt{t} b  ) F_i \right]_{|b=X}  \right].
\end{eqnarray*}
We now write $\phi^{t,b}_i(\cdot)$ to indicate the function on $\Rd$ defined by
$$ \phi^{t,b}_i(F_1,\ldots, F_d) = \pxi \phi( \sqrt{1-t}(F_1,\ldots, F_d) + \sqrt{t} b ) $$
By using Lemma \ref{lem2}, we deduce that
\begin{eqnarray*}
&&\E[\phi^{t,b}_i(F_1,\ldots,F_d)F_i]\\
&&= \E\left[\sum_{j=1}^d \pxj \phi^{t,b}_i(F_1,\ldots,F_d)
\langle DF_j,-DL^{-1}F_i\rangle_{L^2(\mu)}\right] + \E\left[\langle R^i_b,-DL^{-1}F_i\rangle_{L^2(\mu)}\right],
 \end{eqnarray*}
 where $R^i_b$ is a residue verifying
\begin{eqnarray} \label{ineq_res}
&&|\E[\langle R^i_b,-DL^{-1}F_i\rangle_{L^2(\mu)}]|\\
&&\leq \cfrac{1}{2}
\left(\max_{k,l} \sup\limits_{x\in \Rd} \left|\frac{\partial}{\partial x_k \partial x_l} \phi^{t,b}_i (x)\right|\right) \int_Z
 \mu(dz) \E\left[\left(\sum_{j=1}^d|D_z F_j|\right)^2 |D_z L^{-1} F_i|\right].\notag
\end{eqnarray}
 Thus,
\begin{eqnarray*}
\mathfrak{B} &=& \sqrt{1-t} \sum\limits_{i,j=1}^{d} \E\left[ \E \left[\pxij \phi(
\sqrt{1-t}(F_1,\ldots, F_d) + \sqrt{t} b )
 \langle DF_i, -DL^{-1}F_j \rangle_{L^2(\mu)}\right]_{|b=X}  \right]\\
 & & + \sum_{i=1}^d\E\left[ \E \left[\langle  R^i_b,-DL^{-1}F_i  \rangle_{L^2(\mu)}\right]_{|b=X}  \right]\\
 &=& \sqrt{1-t} \sum\limits_{i,j=1}^{d}
\E\left[ \pxij \phi( \sqrt{1-t}(F_1,\ldots, F_d) + \sqrt{t} X)
 \langle DF_i, -DL^{-1}F_j \rangle_{L^2(\mu)} \right]\\
 & & +\sum_{i=1}^d \E\left[ \E \left[\langle  R^i_b,-DL^{-1}F_i  \rangle_{L^2(\mu)}\right]_{|b=X}  \right] .
\end{eqnarray*}
Putting the estimates on $\mathfrak{A}$ and $\mathfrak{B}$ together, we infer
\begin{eqnarray*}
 \Psi '(t)&=& \cfrac{1}{2} \sum\limits_{i,j=1}^{d}
\E\left[ \pxij \phi( \sqrt{1-t}(F_1,\ldots, F_d) + \sqrt{t} X)
(C(i,j) - \langle DF_i, -DL^{-1}F_j \rangle_{L^2(\mu)}) \right] \\
& & - \cfrac{1}{2\sqrt{1-t}}\sum_{i=1}^d \E\left[ \E \left[\langle  R^i_b,-DL^{-1}F_i  \rangle_{L^2(\mu)}\right]_{|b=X}  \right].
\end{eqnarray*}
We  notice that
\begin{eqnarray*}
\left| \pxij \phi( \sqrt{1-t}(F_1,\ldots, F_d) + \sqrt{t} b)\right| \leq \|\phi''\|_{\infty},
 \end{eqnarray*}
and also
\begin{eqnarray*}
\left|\cfrac{\partial^2}{\partial x_k \partial x_l} \phi^{t,b}_i (F_1,\ldots, F_d)\right| & =& (1-t)\times \left|\cfrac{\partial^3}{\partial x_i \partial x_k\partial x_l} \phi( \sqrt{1-t}(F_1,\ldots, F_d) + \sqrt{t} b)\right|
\\ &\leq &(1-t) \|\phi''' \|_{\infty}.
\end{eqnarray*}
To conclude, we can apply inequality (\ref{ineq_res}) and deduce the estimates
\begin{eqnarray*}
& &    |\E[\phi(X)]-\E[\phi(F)]|  \\
&\leq& \sup_{t\in (0,1)} |\Psi '(t)|  \\
&\leq& \cfrac{1}{2} \|\phi'' \|_{\infty}\sum\limits_{i,j=1}^{d} \E[ |C(i,j) - \langle DF_i, -DL^{-1}F_j \rangle_{L^2(\mu)}| ] \\
& & + \cfrac{1-t}{4\sqrt{1-t}}\|\phi'''\|_{\infty} \int_Z
 \mu(dz) \E\left[\left(\smd|D_z F_i|\right)^2 \left(\smd |D_z L^{-1} F_i|\right)\right]\\
&\leq& \cfrac{1}{2} \|\phi'' \|_{\infty} \sqrt{\sum_{i,j=1}^{d} \E[(C(i,j) - \langle  DF_i,-DL^{-1}F_j \rangle_{L^2(\mu)} )^2 ] } \\
& & + \cfrac{1}{4}\|\phi'''\|_{\infty} \int_z
 \mu(dz) \E\left[\left(\smd|D_z F_i|\right)^2 \left(\smd |D_z L^{-1} F_i|\right)\right],
\end{eqnarray*}
thus concluding the proof.
\end{proof}

\smallskip

The following statement is a direct consequence of Theorem \ref{ineq_tala}, as well as a natural generalization of Corollary \ref{Cor1}.

\begin{corollary}\label{Cor2}
For a fixed $d\geq 2$, let $X\sim\mathcal{N}_d(0,C)$, with $C$ a generic covariance matrix. Let $$F_n=(F_{n,1},...,F_{n,d}) = (\hat{N}(h_{n,1}),...,\hat{N}(h_{n,d})),\,\,n\geq 1,$$ be a collection of $d$-dimensional random vectors in the first Wiener chaos of $\hat{N}$, and denote by $K_n$ the covariance matrix of $F_n$. Then,
$$ d_3(F_n,X) \leq \frac12 \|C-K_n\|_{H.S.} + \frac{d^2}{4}\sum_{i=1}^d\int_Z |h_{n,i}(z)|^3\mu(dz). $$
In particular, if relation (\ref{Panocha}) is verified for every $i,j=1,...,d$ (as $n\rightarrow\infty$), then $d_3(F_n,X) \rightarrow 0$ and $F_n$ converges in distribution to $X$.
\end{corollary}

\subsection{Stein's method versus smart paths: two tables}\label{SS : tables}

In the two tables below, we compare the estimations obtained by the Malliavin-Stein method with those deduced by interpolation techniques, both in a Gaussian and Poisson setting. Note that the test functions considered below have (partial) derivatives that are not necessarily bounded by 1 (as it is indeed the case in the definition of the distances $d_2$ and $d_3$) so that the $L^\infty$ norms of various derivatives appear in the estimates. In both tables, $d\geq 2$ is a given positive integer. We write $(G,G_1,\ldots,G_d)$ to indicate a vector of centered Malliavin differentiable functionals of an isonormal Gaussian process over some separable real Hilbert space $\h$ (see \cite{nualart} for definitions). We write $(F,F_1,...,F_d)$ to indicate a vector of centered functionals of $\hat{N}$, each belonging to ${\rm dom} D$. The symbols $D$ and $L^{-1}$ stand for the Malliavin derivative and the inverse of the Ornstein-Uhlenbeck generator: plainly, both are to be regarded as defined either on a Gaussian space or on a Poisson space, according to the framework. We also consider the following Gaussian random elements: $X\sim \N(0,1)$, $X_C \sim \N_d(0,C)$ and $X_M \sim \N_d(0,M)$, where $C$ is a $d\times d$ positive definite covariance matrix and $M$ is a $d\times d$ covariance matrix (not necessarily positive definite).\\

 In Table 1, we present all estimates on distances involving Malliavin differentiable random variables (in both cases of an underlying Gaussian and Poisson space), that have been obtained by means of Malliavin-Stein techniques. These results are taken from: \cite{np} (Line 1), \cite{npr} (Line 2), \cite{pstu} (Line 3) and Theorem \ref{ineq_2W} and its proof (Line 4).\\

\begin{table}
\caption{Estimates proved by means of Malliavin-Stein techniques}
\centering
\begin{tabular}{|c|c|}
  \hline
  {\bf Regularity of}&  {\bf Upper bound} \\
  {\bf the test function} $h$ &  \\
  \hline
  & \\
  $\|h\|_{Lip}$ is finite & $|\E[h(G)] - \E[h(X)] | \leq $ \\
   &$\|h\|_{Lip} \sqrt{\E[(1 - \langle  {D}G,-{D}{L}^{-1}G \rangle_{\h} )^2 ] }$ \\
  \hline
& \\
  $\|h\|_{Lip}$ is finite  & $|\E[h(G_1,\ldots,G_d)] - \E[h(X_C)] | \leq $ \\
  & $\|h\|_{Lip} \|C^{-1}\|_{op} \|C\|_{op}^{1/2} \sqrt{\sum_{i,j=1}^{d} \E[(C(i,j) - \langle  {D}G_i,-{D}{L}^{-1}G_j \rangle_{\h} )^2 ] }$  \\
  \hline
& \\
    $\|h\|_{Lip}$ is finite &   $|\E[h(F)] - \E[h(X)] | \leq $ \\
   & $\|h\|_{Lip} ( \sqrt{\E[(1 - \langle  DF,-DL^{-1}F \rangle_{L^2(\mu)} )^2 ] }$  \\
    &$+ \int_Z \mu(dz)\E[(|D_z F | )^2 |D_z L^{-1} F |  ])$ \\
  \hline
& \\
  $h \in \C^2(\Rd)$&  $|\E[h(F_1,\ldots,F_d)] - \E[h(X_C)] | \leq $\\
  $\|h\|_{Lip}$ is finite &  $\|h\|_{Lip} \|C^{-1}\|_{op} \|C\|_{op}^{1/2} \sqrt{\sum_{i,j=1}^{d} \E[(C(i,j) - \langle  DF_i,-DL^{-1}F_j \rangle_{L^2(\mu)} )^2 ] }$  \\
  $M_2(h)$ is finite &$+ M_2(h) \cfrac{\sqrt{2\pi}}{8} \|C^{-1}\|^{3/2}_{op} \|C\|_{op}  \int_Z \mu(dz)\E\left[\left(\smd|D_z F_i | \right)^2 \left(\smd|D_z L^{-1} F_i |  \right)
\right] $
\\
  \hline
\end{tabular}
\end{table}

In Table 2, we list the parallel results obtained by interpolation methods. The bounds involving functionals of a Gaussian process come from \cite{npr2}, whereas those for Poisson functionals are taken from Theorem \ref{ineq_tala} and its proof.\\

\begin{table}
\caption{Estimates proved by means of interpolations}
\centering
\begin{tabular}{|c|c|}
  \hline
  {\bf Regularity of}& {\bf Upper bound}\\
  {\bf the test function} $\phi$&  \\
  \hline
& \\
  $\phi\in \C^2(\R)$&  $|\E[\phi(G)] - \E[\phi(X)] | \leq $ \\
  $\|\phi''\|_{\infty}$ is finite &$\frac{1}{2} \|\phi''\|_{\infty} \sqrt{\E[(1 - \langle  {D}G,-{D}{L}^{-1}G \rangle_{\h} )^2 ] }$ \\
  \hline
& \\
  $\phi\in \C^2(\Rd)$& $|\E[\phi(G_1,\ldots,G_d)] - \E[\phi(X_M)] | \leq $\\
  $\|\phi''\|_{\infty}$ is finite & $\frac{1}{2} \|\phi''\|_{\infty} \sqrt{\sum_{i,j=1}^{d} \E[(M(i,j) - \langle  {D}G_i,-{D}{L}^{-1}G_j \rangle_{\h} )^2 ] }$  \\
  \hline
& \\
  $\phi\in \C^3(\R)$  & $|\E[\phi(F)] - \E[\phi(X)] | \leq $ \\
  $\|\phi''\|_{\infty}$ is finite & $\frac{1}{2} \|\phi''\|_{\infty}  \sqrt{\E[(1 - \langle  DF,-DL^{-1}F \rangle_{L^2(\mu)} )^2 ] }$  \\
  $\|\phi'''\|_{\infty}$ is finite  &$+ \frac{1}{4} \|\phi'''\|_{\infty} \int_Z \mu(dz)\E[(|D_z F | )^2 (|D_z L^{-1} F |  )]$ \\
  \hline
& \\
  $\phi \in \C^3(\Rd)$&  $|\E[\phi(F_1,\ldots,F_d)] - \E[\phi(X_M)] | \leq $\\
  $\|\phi''\|_{\infty}$ is finite &  $\frac{1}{2} \|\phi''\|_{\infty} \sqrt{\sum_{i,j=1}^{d} \E[(M(i,j) - \langle  DF_i,-DL^{-1}F_j \rangle_{L^2(\mu)} )^2 ] }$  \\
  $ \|\phi'''\|_{\infty}$ is finite &$+ \frac{1}{4}\|\phi'''\|_{\infty} \int_Z \mu(dz)\E\left[\left(\smd|D_z F_i | \right)^2 \left(\smd|D_z L^{-1} F_i | \right)
\right]$
\\
  \hline
\end{tabular}
\end{table}

Observe that:
\begin{itemize}
\item in contrast to the Malliavin-Stein method, the covariance matrix $M$ is not required to be positive definite when using the interpolation technique,
  \item in general, the interpolation technique requires more regularity on test functions than the Malliavin-Stein method.
\end{itemize}


\end{section}

\begin{section}{CLTs for Poisson multiple integrals}
In this section, we study the Gaussian approximation of vectors of Poisson multiple stochastic integrals by an application of Theorem \ref{ineq_2W} and Theorem \ref{ineq_tala}. To this end, we shall explicitly evaluate the quantities appearing in formulae (\ref{L1})--(\ref{L2}) and (\ref{ineq_tala2})--(\ref{unotrino2}).

\begin{remark}[Regularity conventions]\label{RMK : CONVENTIONS}
{\rm From now on, every kernel $f\in L^2_s(\mu ^p)$ is supposed to verify both Assumptions A and B of Definition \ref{DEF : ASSmptns}. As before, given $f\in L_s^2(\mu^p)$, and for a fixed $z\in Z$, we write $f(z,\cdot)$ to indicate the function defined on $Z^{p-1}$ as $(z_1,\ldots,z_{p-1}) \mapsto f(z,z_1,\ldots,z_{p-1}) $.  The following convention will be also in order: given a vector of kernels $(f_1,...,f_d)$ such that $f_i\in L^2_s(\mu^{p_i})$, $i=1,...,d$, we will implicitly set $$f_i(z,\cdot) \equiv 0, \,\,\, i=1,...,d,$$ for every $z\in Z$ belonging to the exceptional set (of $\mu$ measure 0) such that $$f_i(z,\cdot)\star_r^l f_j(z,\cdot) \in \!\!\!\!\! / \, L^2(\mu^{p_i+p_j-r-l-2} )$$ for at least one pair $(i,j)$ and some $r=0,...,p_i\wedge p_j-1$ and $l=0,...,r$. See Point 3 of Remark \ref{REMK : ASSUMPTIONS}.}
\end{remark}

\begin{subsection}{The operators $G_k^{p,q}$ and $\widehat{G_k^{p,q}}$}
Fix integers $p,q \geq 0$ and $|q-p| \leq k \leq p+q$, consider two kernels $f\in L^2_s(\mu^p)$ and $g \in L^2_s(\mu^q)$, and recall the multiplication formula (\ref{product}). We will now introduce an operator
$G_k^{p,q}$, transforming the function $f$, of $p$ variables,
and the function $g$, of $q$ variables, into a ``hybrid'' function $G_k^{p,q}(f,g)$, of
$k$ variables. More precisely, for $p,q,k $ as above, we define the function $(z_1,\ldots,z_k)\mapsto G_k^{p,q}(f,g)(z_1,\ldots,z_k)$, from $Z^k$ into $\R$, as follows:
\begin{equation}\label{force}
G_k^{p,q}(f,g)(z_1,\ldots,z_k) = \sum_{r=0}^{p\wedge q} \sum_{l=0}^{r}
\1_{(p+q-r-l=k)} r!
\left(
\begin{array}{c}
  p\\
  r\\
\end{array}
\right)
 \left(
\begin{array}{c}
  q\\
  r\\
\end{array}
\right)
 \left(
\begin{array}{c}
  r\\
  l\\
\end{array}
\right)  \widetilde{f\star_r^l g} ,
\end{equation}
where the tilde $\sim$ means symmetrization, and the star contractions are defined in formula (\ref{contraction}) and the subsequent discussion. Observe the following three special cases: (i) when $p=q=k=0$, then $f$ and $g$ are both real constants, and $G_0^{0,0}(f,g) = f\times g$, (ii) when $ p=q\geq 1$ and $k=0$, then $G_0^{p,p}(f,g) = p!\langle f,g \rangle_{L^2(\mu^p)} $, (iii) when $p=k=0$ and $q>0$ (then, $f$ is a constant), $G_0^{0,p}(f,g)(z_1,...,z_q) = f\times g(z_1,...,z_q)$. By using this notation, (\ref{product}) becomes
\begin{equation} \label{product2}
I_p(f)I_q(g) = \sum_{k=|q-p|}^{p+q} I_k(G_k^{p,q}(f,g)).
\end{equation}
The advantage of representation (\ref{product2}) (as opposed to (\ref{product})) is
that the RHS of (\ref{product2}) is an \emph{orthogonal sum}, a feature that will greatly simplify our forthcoming computations. \\

For two functions $f\in L_s^2(\mu^p) $ and $g\in L_s^2(\mu^q) $, we define the function $(z_1,\ldots,z_k)\mapsto \widehat{G_k^{p,q}}(f,g) (z_1,\ldots,z_k)$, from $Z^k$ into $\R$, as follows:
$$\widehat{G_k^{p,q}}(f,g) (\cdot) = \int_Z \mu(dz) {G_k^{p-1,q-1}}(f(z,\cdot),g(z,\cdot)),    $$
or, more precisely,
\begin{eqnarray}
& & \widehat{G_k^{p,q}}(f,g) (z_1,\ldots,z_k) \nonumber \\
&=&\!\!\! \int_Z \mu(dz) \sum_{r=0}^{p\wedge q -1} \sum_{l=0}^{r}
\1_{(p+q-r-l-2=k)} r!  \nonumber \\
&&\quad\quad\quad\quad\quad\quad\quad
\times \left(
\begin{array}{c}
  p-1\\
  r\\
\end{array}
\right)
 \left(
\begin{array}{c}
  q-1\\
  r\\
\end{array}
\right)
 \left(
\begin{array}{c}
  r\\
  l\\
\end{array}
\right)  \widetilde{f(z,\cdot)\star_r^l g(z,\cdot)} (z_1,\ldots,z_k)  \nonumber  \\
&=&\!\!\!  \sum_{t=1}^{p\wedge q} \sum_{s=1}^{t}
\1_{(p+q-t-s=k)} (t-1)!
\left(
\begin{array}{c}
  p-1\\
  t-1\\
\end{array}
\right)
\! \left(
\begin{array}{c}
  q-1\\
  t-1\\
\end{array}
\right)
\! \left(
\begin{array}{c}
  t-1\\
  s-1\\
\end{array}
\right)  \widetilde{f\star_t^s g} (z_1,\ldots,z_k) .  \label{G_hat}
\end{eqnarray}
Note that the implicit use of a Fubini theorem in the equality (\ref{G_hat}) is justified by Assumption B -- see again Point 3 of Remark \ref{REMK : ASSUMPTIONS}.\\

The following technical lemma will be applied in the next subsection.

\begin{lemma} \label{lemma_G_hat}
Consider three positive integers $p,q,k$ such that $p,q\geq 1$ and $|q-p| \vee 1 \leq k \leq p+q-2 $. For any two kernels $f\in L^2_s(\mu^p)$ and $ g\in L^2_s(\mu^q)$, both verifying Assumptions A and B, we have
\begin{eqnarray}
\int_{Z^k} d\mu^k (\widehat{G_k^{p,q}}(f,g)(z_1,\ldots,z_k))^2 \leq C \sum_{t=1}^{p\wedge q} \1_{1\leq s(t,k) \leq t}
\| \widetilde{f \star_t^{s(t,k)} g}\|_{L^2(\mu^k)}^2
\end{eqnarray}
where $s(t,k)=p+q-k-t$ for $t=1,\ldots,p\wedge q$. Also, $C$ is the constant given by
$$ C=\sum_{t=1}^{p\wedge q} \left[(t-1)!
\left(
\begin{array}{c}
  p-1\\
  t-1\\
\end{array}
\right)
 \left(
\begin{array}{c}
  q-1\\
  t-1\\
\end{array}
\right)
 \left(
\begin{array}{c}
  t-1\\
  s(t,k)-1\\
\end{array}
\right) \right]^2 . $$
\end{lemma}
\begin{proof}
We rewrite the sum in (\ref{G_hat}) as
\begin{equation} \label{G_hat2}
 \widehat{G_k^{p,q}}(f,g) (z_1,\ldots,z_k) =  \sum_{t=1}^{p\wedge q} a_t  \1_{1\leq s(t,k) \leq t} \widetilde{f\star_t^{s(t,k)} g} (z_1,\ldots,z_k) ,
\end{equation}
with $a_t = (t-1)!
\left(
\begin{array}{c}
  p-1\\
  t-1\\
\end{array}
\right)
 \left(
\begin{array}{c}
  q-1\\
  t-1\\
\end{array}
\right)
 \left(
\begin{array}{c}
  t-1\\
  s(t,k)-1\\
\end{array}
\right)$, $\quad 1\leq t \leq p\wedge q$.
Thus,
\begin{eqnarray*}
& &\int_{Z^k} d\mu^k (\widehat{G_k^{p,q}}(f,g)(z_1,\ldots,z_k))^2 \\
&=& \int_{Z^k} d\mu^k \left( \sum_{t=1}^{p\wedge q} a_t  \1_{1\leq s(t,k) \leq t} \widetilde{f\star_t^{s(t,k)} g} (z_1,\ldots,z_k) \right)^2 \\
&\leq& \left(\sum_{t=1}^{p\wedge q} a_t^2\right) \int_{Z^k} d\mu^k  \left(\sum_{t=1}^{p\wedge q} (\1_{1\leq s(t,k) \leq t} \widetilde{f \star_t^{s(t,k)} g}(z_1,\ldots,z_k) )^2\right) \\
&=&   C \sum_{t=1}^{p\wedge q} \int_{Z^k} d\mu^k \1_{1\leq s(t,k) \leq t} (\widetilde{f \star_t^{s(t,k)} g}(z_1,\ldots,z_k) )^2 \\
&=& C \sum_{t=1}^{p\wedge q} \1_{1\leq s(t,k) \leq t}
\| \widetilde{f \star_t^{s(t,k)} g}\|_{L^2(\mu^k)}^2,
\end{eqnarray*}
with
$$ C=\sum_{t=1}^{p\wedge q} a_t^2=\sum_{t=1}^{p\wedge q} \left[(t-1)!
\left(
\begin{array}{c}
  p-1\\
  t-1\\
\end{array}
\right)
 \left(
\begin{array}{c}
  q-1\\
  t-1\\
\end{array}
\right)
 \left(
\begin{array}{c}
  t-1\\
  s(t,k)-1\\
\end{array}
\right) \right]^2 $$
Note that the Cauchy-Schwarz inequality
$$ \left(\sum_{i=1}^{n} a_i x_i \right)^2 \leq \left(\sum_{i=1}^{n} a_i^2 \right) \left(\sum_{i=1}^{n} x_i^2 \right) $$ has been used in the above deduction.
\end{proof}

\end{subsection}

\begin{subsection}{Some technical estimates}
As anticipated, in order to prove the multivariate CLTs of the forthcoming Section \ref{subsection_CLT}, we need to establish explicit bounds on the quantities appearing in (\ref{L1})--(\ref{L2}) and (\ref{ineq_tala2})--(\ref{unotrino2}), in the special case of chaotic random variables. \\

\begin{definition}
Let the integers $p,q\geq 1$ be such that $\max(p,q)>1$. The kernels $f\in L^2_s(\mu^p)$, $g\in L^2_s(\mu^q)$ are said to satisfy {\bf Assumption C} if, for every $k=|q-p|\vee 1, \ldots, p+q-2$,
\begin{equation}  \label{Fubini_cond}
\int_Z \left[\sqrt{\int_{Z^k}(G_k^{p-1,q-1} (f(z,\cdot),g(z,\cdot)) )^2 d \mu^k } \quad \right] \mu(dz) < \infty .
\end{equation}
\end{definition}

\begin{remark}\label{RMK : ON CONDITIONS}{\rm
By using (\ref{force}), one sees that (\ref{Fubini_cond}) is implied by the following stronger condition: for every $k=|q-p|\vee 1, \ldots, p+q-2$, and every $(r,l)$ satisfying $p+q-2-r-l=k$, one has
\begin{equation}  \label{Fubini_cond2}
\int_Z \left[\sqrt{\int_{Z^k}(f(z,\cdot)\star^l_rg(z,\cdot))^2 d \mu^k } \quad \right] \mu(dz) < \infty .
\end{equation}
One can easily write down sufficient conditions, on $f$ and $g$, ensuring that (\ref{Fubini_cond2}) is satisfied. For instance, in the examples of Section 6, we will use repeatedly the following fact: if both $f$ and $g$ verify Assumption A, and if their supports are contained in the finite union of rectangles of the type $B\times \ldots \times B$, with $\mu(B)<\infty$, then (\ref{Fubini_cond2}) is automatically satisfied.
}
\end{remark}


\begin{proposition} \label{bound}
Denote by $L^{-1}$ the pseudo-inverse of the Ornstein-Uhlenbeck generator (see the Appendix in Section \ref{APPENDIX}), and let $F=I_p(f)$ and $G=I_q(g)$ be such that the kernels $f\in L^2_s(\mu^p)$ and $g \in L^2_s(\mu^q)$ verify Assumptions A, B and C. If $p \neq q$, then
\begin{eqnarray*}
& & \E[(a - \langle DF,-DL^{-1}G \rangle_{L^2(\mu)}  )^2 ] \\
&\leq& a^2 + p^2 \sum_{k=|q-p|}^{p+q-2} k! \int_{Z^k} d\mu^k (\widehat{G_k^{p,q}}(f,g))^2 \\
&\leq& a^2 + C p^2 \sum_{k=|q-p|}^{p+q-2} k! \sum_{t=1}^{p\wedge q} \1_{1\leq s(t,k) \leq t} \| \widetilde{f \star_t^{s(t,k)} g}\|_{L^2(\mu^k)}^2 \\
&\leq& a^2 + \cfrac{1}{2} C p^2 \sum_{k=|q-p|}^{p+q-2} k! \sum_{t=1}^{p\wedge q} \1_{1\leq s(t,k) \leq t} (\|f\star^{p-t}_{p-s(t,k)}f\|_{L^2(\mu^{t+s(t,k)})} \times
 \|g\star^{q-t}_{q-s(t,k)}g\|_{L^2(\mu^{t+s(t,k)})} ) \\
\end{eqnarray*}
If $p = q$, then
\begin{eqnarray*}
& & \E[(a - \langle DF,-DL^{-1}G \rangle_{L^2(\mu)}  )^2 ] \\
&\leq & (p!\langle f,g \rangle_{L^2(\mu^p)} -a)^2 + p^2 \sum_{k=1}^{2p-2}
k! \int_{Z^k} d\mu^k (\widehat{G_k^{p,q}}(f,g))^2  \\
&\leq & (p!\langle f,g \rangle_{L^2(\mu^p)} -a)^2 + C p^2 \sum_{k=1}^{2p-2}
k! \sum_{t=1}^{p\wedge q} \1_{1\leq s(t,k) \leq t}
\| \widetilde{f \star_t^{s(t,k)} g}\|_{L^2(\mu^k)}^2 \\
&\leq & (p!\langle f,g \rangle_{L^2(\mu^p)} -a)^2 \\
&& + \cfrac{1}{2} C p^2 \sum_{k=1}^{2p-2}
k! \sum_{t=1}^{p\wedge q} \1_{1\leq s(t,k) \leq t}  (\|f\star^{p-t}_{p-s(t,k)}f\|_{L^2(\mu^{t+s(t,k)})} \times
 \|g\star^{q-t}_{q-s(t,k)}g\|_{L^2(\mu^{t+s(t,k)})} ) \\
\end{eqnarray*}
where $s(t,k)=p+q-k-t$ for $t=1,\ldots,p\wedge q$. Finally, the constant $C$ is given by
$$ C=\sum_{t=1}^{p\wedge q} \left[(t-1)!
\left(
\begin{array}{c}
  p-1\\
  t-1\\
\end{array}
\right)
 \left(
\begin{array}{c}
  q-1\\
  t-1\\
\end{array}
\right)
 \left(
\begin{array}{c}
  t-1\\
  s(t,k)-1\\
\end{array}
\right) \right]^2.
$$
\end{proposition}
\begin{proof}
We select two versions of the derivatives $D_zF=pI_{p-1}(f(z,\cdot))$ and $D_zG=qI_{q-1}(g(z,\cdot))$, in such a way that the conventions pointed out in Remark \ref{RMK : CONVENTIONS} are satisfied. By using the definition of $L^{-1}$ and (\ref{product2}), we have
\begin{eqnarray*}
\langle DF,-DL^{-1}G \rangle_{L^2(\mu)} &=&\langle D I_p(f), q^{-1
}D I_q(g) \rangle_{L^2(\mu)}    \\
&=&p \int_Z\mu(dz) I_{p-1}(f(z,\cdot)) I_{q-1}(g(z,\cdot))  \\
&=& p \int_Z \mu(dz) \sum_{k=|q-p|}^{p+q-2} I_k(G_k^{p-1,q-1}(f(z,\cdot),g(z,\cdot))) \\
\end{eqnarray*}
Notice that for $i \neq j$, the two random variables
$$\int_Z \mu(dz) I_i(G_i^{p-1,q-1}(f(z,\cdot),g(z,\cdot))\quad \text{and}
\quad \int_Z \mu(dz) I_j(G_j^{p-1,q-1}(f(z,\cdot),g(z,\cdot))) $$
are orthogonal in $L^2(\p)$. It follows that
\begin{eqnarray}
& & \E[(a - \langle DF,-DL^{-1}G \rangle_{L^2(\mu)}  )^2 ] \label{volks}\\
&&\quad\quad\quad\quad= a^2 + p^2 \sum_{k=|q-p|}^{p+q-2} \E\left[\left(\int_Z \mu(dz) I_k(G_k^{p-1,q-1}(f(z,\cdot),g(z,\cdot)))    \right)^2\right] \notag
\end{eqnarray}
for $p \neq q$, and, for $p=q$,
\begin{eqnarray}
& & \E[(a - \langle DF,-DL^{-1}G \rangle_{L^2(\mu)}  )^2 ] \label{wagen}\\
&&\quad\quad= (p!\langle f,g \rangle_{L^2(\mu^p)} -a)^2 + p^2 \sum_{k=1}^{2p-2} \E\left[\left(\int_Z \mu(dz) I_k(G_k^{p-1,q-1}(f(z,\cdot),g(z,\cdot)))    \right)^2\right]. \notag
\end{eqnarray}
We shall now assess the expectations appearing on the RHS of (\ref{volks}) and (\ref{wagen}). To do this, fix an integer $k$ and use the Cauchy-Schwartz inequality together with (\ref{Fubini_cond}) to deduce that
\begin{eqnarray}
& &\int_Z \mu(dz) \int_Z \mu(dz') \E\left[ \left| I_k(G_k^{p-1,q-1}(f(z,\cdot),g(z,\cdot)))
I_k(G_k^{p-1,q-1}(f(z',\cdot),g(z',\cdot))) \right| \right] \notag \\
&\leq& \int_Z \mu(dz) \int_Z \mu(dz')
\sqrt{ \E[ I_k^2(G_k^{p-1,q-1}(f(z,\cdot),g(z,\cdot)))    ]}
\sqrt{ \E[ I_k^2(G_k^{p-1,q-1}(f(z',\cdot),g(z',\cdot)))  ]} \notag\\
&=&  k! \left[\int_Z \mu(dz) \sqrt{\int_{Z^k} d\mu^k (G_k^{p-1,q-1}(f(z,\cdot),g(z,\cdot)))^2} \quad\right] \notag\\
& & \quad\quad\quad\quad\quad\quad\quad\quad\times \left[\int_Z \mu(dz') \sqrt{\int_{Z^k} d\mu^k (G_k^{p-1,q-1}(f(z',\cdot),g(z',\cdot)))^2} \quad\right] \notag\\
&=&  k! \left[\int_Z \mu(dz) \sqrt{\int_{Z^k} d\mu^k (G_k^{p-1,q-1}(f(z,\cdot),g(z,\cdot)))^2} \quad \right]^2 < \infty. \label{milan}
\end{eqnarray}
Relation (\ref{milan}) justifies the use of a Fubini theorem, and we can consequently infer that
\begin{eqnarray*}
& &  \E\left[\left(\int_Z \mu(dz) I_k(G_k^{p-1,q-1}(f(z,\cdot),g(z,\cdot)))    \right)^2\right] \\
&=& \int_Z \mu(dz) \int_Z \mu(dz') \E[ I_k(G_k^{p-1,q-1}(f(z,\cdot),g(z,\cdot)))
I_k(G_k^{p-1,q-1}(f(z',\cdot),g(z',\cdot))) ]  \\
&=& k! \int_Z \mu(dz) \int_Z \mu(dz') \left[\int_{Z^k} d\mu^k {G_k^{p-1,q-1}}(f(z,\cdot),g(z,\cdot)) {G_k^{p-1,q-1}}(f(z',\cdot),g(z',\cdot)) \right]  \\
&=& k! \int_{Z^k} d\mu^k \left[\int_{Z} \mu(dz) {G_k^{p-1,q-1}}(f(z,\cdot),g(z,\cdot)) \right]^2  \\
&=& k! \int_{Z^k} d\mu^k (\widehat{G_k^{p,q}}(f,g))^2 .
\end{eqnarray*}
The remaining estimates in the statement follow (in order) from Lemma \ref{lemma_G_hat} and Lemma \ref{lemma_ineq}, as well as from the fact that
$\|\widetilde{f}\|_{L^2(\mu^n)} \leq  \|f\|_{L^2(\mu^n)} $, for all $n\geq 2$.
\end{proof}\\

The next statement will be used in the subsequent section.
\begin{proposition} \label{ineq_res2}
Let $ F = (F_1,\ldots , F_d) := ( I_{q_1}(f_1),\ldots , I_{q_d}(f_d)) $ be a vector of Poisson functionals, such that the kernels $f_j$ verify Assumptions A and B. Then, by noting $q_{*} := \min\{q_1,...,q_d\}$,
\begin{eqnarray*}
& &\int_Z \mu(dz)\E\left[\left(\smd|D_z F_i | \right)^2 \left(\smd|D_z L^{-1} F_i |  \right) \right] \\
&\leq& \frac{d^2}{q_{*}} \smd
\Big( q_i^3 \sqrt{(q_i-1)!\|f\|^2_{L^2(\mu^{q_i})}}
\times \sum_{b=1}^{q_i} \sum_{a=0}^{b-1} \1_{1\leq a+b\leq 2q_i-1} (a+b-1)!^{1/2} (q_i-a-1)!\\
& & \times
\left(
\begin{array}{c}
  q_i-1\\
  q_i-1-a\\
\end{array}
\right)^2
 \left(
\begin{array}{c}
  q_i-1-a\\
  q_i-b\\
\end{array}
\right)
\|f\star^a_b f \|_{L^2(\mu^{2q_i-a-b})} \Big).
\end{eqnarray*}
\end{proposition}
\begin{proof} One has that
\begin{eqnarray*}
&&\int_Z \mu(dz)\E\left[\left(\smd|D_z F_i | \right)^2 \left(\smd|D_z L^{-1} F_i |  \right) \right] \\
&&=\int_Z \mu(dz)\E\left[\left(\smd|D_z F_i | \right)^2 \left(\smd \frac{1}{q_i}|D_z F_i |  \right) \right] \\
&&\leq \frac{1}{q_{*}} \int_Z \mu(dz)\E\left[\left(\smd|D_z F_i | \right)^3 \right] \\
&&\leq \frac{d^2}{q_{*}} \smd \int_Z \mu(dz)\E[|D_z F_i |^3 ]. \\
\end{eqnarray*}
To conclude, use the inequality
\begin{eqnarray*}
& & \int_Z \mu(dz) |\E[D_z I_q(f)]|^3  \\
&\leq& q^3 \sqrt{(q-1)!\|f\|^2_{L^2(\mu^q)}}
\times \sum_{b=1}^{q} \sum_{a=0}^{b-1} \1_{1\leq a+b\leq 2q-1} (a+b-1)!^{1/2} (q-a-1)!\\
& & \times
\left(
\begin{array}{c}
  q-1\\
  q-1-a\\
\end{array}
\right)^2
 \left(
\begin{array}{c}
  q-1-a\\
  q-b\\
\end{array}
\right)
\|f\star^a_b f \|_{L^2(\mu^{2q-a-b})}
\end{eqnarray*}
which is proved in \cite[Theorem 4.2]{pstu} (see in particular formulae (4.13) and (4.18) therein).
\end{proof}

\end{subsection}

\begin{subsection}{Central limit theorems with contraction conditions}  \label{subsection_CLT}
We will now deduce the announced CLTs for sequences of vectors of the type
\begin{equation}\label{redriding}
F^{(n)} = (F^{(n)}_1,\ldots , F^{(n)}_d) := ( I_{q_1}(f^{(n)}_1),\ldots , I_{q_d}(f^{(n)}_d)), \quad n\geq 1.
\end{equation}
 As already discussed, our results should be compared with other central limit results for multiple stochastic integrals in a Gaussian or Poisson setting -- see e.g. \cite{np, npr, no, NuPe, pt, ptudor}. The following statement, which is a genuine multi-dimensional generalization of Theorem 5.1 in \cite{pstu}, is indeed one of the main achievements of the present article.

\begin{theorem}[CLT for chaotic vectors] \label{CLT}
Fix $d\geq 2$, let $X \sim \N(0,C)$, with $$C=\{C(i,j):i,j = 1,\ldots,d \}$$ a $d\times d$ nonnegative definite matrix, and fix integers $q_1 , \ldots , q_d\geq 1$. For any $n\geq 1$
and $i=1,\ldots,d$, let $f_i^{(n)}$ belong to $ L^2_s(\mu^{q_i})$. Define the sequence $\{F^{(n)}:n\geq 1\}$, according to (\ref{redriding}) and suppose that
\begin{equation} \label{covariance}
 \lim_{n\rightarrow \infty} \E[F^{(n)}_i F^{(n)}_j ]={\bf 1}_{(q_j=q_i)}\times\lim_{n\rightarrow \infty}\langle f_i^{(n)},f_j^{(n)}\rangle_{L^2(\mu^{q_i})} =C(i,j),\qquad 1\leq i,j \leq d.
\end{equation}
Assume moreover that the following Conditions 1--4 hold for every $k=1,...,d$:
\begin{enumerate}
  \item For every $n$, the kernel $f^{(n)}_k$ satisfies Assumptions A and B.
  \item For every $l=1,...,d$ and every $n$, the kernels $f^{(n)}_k$ and $f^{(n)}_l$ satisfy Assumption C.
  \item For every $r=1,\ldots,q_k$ and every $l=1,\ldots, r\wedge (q_k-1)$, one has
  that $$\|f^{(n)}_k \star_r^l f^{(n)}_k \|_{L^2(\mu^{2 q_k-r-l})} \rightarrow 0,$$ as $n\rightarrow \infty$.
  \item As $n\rightarrow \infty$, $\int_{Z^{q_k}}d\mu^{q_k}\left(f^{(n)}_k\right)^4 \rightarrow 0 $.
\end{enumerate}
Then, $ F^{(n)} $ converges to $X$ in distribution as $n\rightarrow \infty$. The speed of convergence can be assessed by combining the estimates of Proposition \ref{bound} and Proposition \ref{ineq_res2} either with Theorem \ref{ineq_2W} (when $C$ is positive definite) or with Theorem \ref{ineq_tala} (when $C$ is merely nonnegative definite).
\end{theorem}
\begin{proof}
By Theorem \ref{ineq_tala},
\begin{eqnarray}
d_3(F^{(n)},X) \leq \frac12 \sqrt{\sum_{i,j=1}^{d} \E[(C(i,j) - \langle  DF^{(n)}_i,-DL^{-1}F^{(n)}_j \rangle_{L^2(\mu)} )^2 ] } \label{part1}\\
+ \frac14 \int_Z
\mu(dz)\E\left[ \left(\smd|D_z F^{(n)}_i | \right)^2 \left(\smd|D_z L^{-1} F^{(n)}_i |  \right) \right],    \label{part2}
\end{eqnarray}
so that we need only show that, under the assumptions in the statement, both (\ref{part1}) and (\ref{part2}) tend to $0$ as $n\rightarrow \infty$. That (\ref{part1}) tends to $0$ is a direct consequence of the estimates in Proposition \ref{bound}, whereas Proposition \ref{ineq_res2} shows that (\ref{part2}) converges to $0$. This concludes the proof. \\

If the matrix $C$ is positive definite, then one could alternatively use Theorem \ref{ineq_2W} instead of Theorem \ref{ineq_tala}.
\end{proof}

\begin{remark}
\rm{ Apart from the asymptotic behavior of the covariances (\ref{covariance}) and the presence of Assumption C, the statement of Theorem \ref{CLT} does not contain any requirements on the joint distribution of the components of $F^{(n)}$. Besides the technical requirements in Condition 1 and Condition 2, the joint convergence of the random vectors $F^{(n)}$ only relies on the `one-dimensional' Conditions 3 and 4, which are the same as condition (II) and (III) in the statement of Theorem 5.1 in \cite{pstu}. See also Remark \ref{rmk2}.
}
\end{remark}

\end{subsection}
\end{section}

\begin{section}{Examples}
In what follows, we provide several explicit applications of the main estimates proved in the paper. In particular:
\begin{itemize}
  \item Section \ref{section_vector12} focuses on vectors of single and double integrals.
  \item Section \ref{section_OU_ex} deals with three examples of continuous-time functionals of Ornstein-Uhlenbeck L\'{e}vy processes.
\end{itemize}
\begin{subsection}{Vectors of single and double integrals} \label{section_vector12}
The following statement corresponds to Theorem \ref{ineq_2W}, in the special case
\begin{equation}\label{scriabin}
F=(F_1,\ldots,F_d) = (I_1(g_1),\ldots, I_1(g_m),I_2(h_1),\ldots, I_2(h_n)) .
\end{equation}
The proof, which is based on a direct computation of the general bounds proved in Theorem \ref{ineq_2W}, serves as a further illustration (in a simpler setting) of the techniques used throughout the paper. Some of its applications will be illustrated in Section \ref{section_OU_ex}.
\begin{proposition} \label{example12} Fix integers $n,m\geq 1$, let $d=n+m$, and let $C$ be a $d\times d$ nonnegative definite matrix. Let $X\sim \mathcal{N}_d(0,C)$.
Assume that the vector in (\ref{scriabin}) is such that
\begin{enumerate}
    \item the function $g_i$
belongs to $L^2(\mu) \cap L^3(\mu)$, for every $1\leq i \leq m$,
    \item the kernel $h_i \in L^2_s(\mu^2)$ ($1\leq i\leq n$) is such that: {\rm (a)} $h_{i_1} \star^1_2 h_{i_2} \in L^2(\mu^1) $, for $1\leq i_1,i_2 \leq n$, {\rm (b)} $h_i\in L^4(\mu^2) $ and {\rm (c)} the functions $|h_{i_1}| \star^1_2 |h_{i_2}|$, $|h_{i_1}| \star^0_2 |h_{i_2}|$ and $|h_{i_1}| \star^0_1 |h_{i_2}|$ are well defined and finite for every value of their arguments and for every $1\leq i_1,i_2 \leq n$, {\rm (d)} every pair $(h_i,\, h_j)$ verifies Assumption C, that in this case is equivalent to requiring that $$\int_Z\sqrt{\int_Z \mu(da) h^2_i(z,a)h^2_j(z,a)}\mu(dz)<\infty.$$
\end{enumerate}
Then,

\begin{eqnarray*}
d_{3}(F,X) &\leq&   \cfrac{1}{2} \sqrt{S_1 + S_2 + S_3}  +   S_4  \\
&\leq&    \cfrac{1}{2} \sqrt{S_1 + S_5 + S_6}  +   S_4
\end{eqnarray*}

where
\begin{eqnarray*}
S_1 &=& \sum_{i_1,i_2=1}^{m} (C(i_1,i_2)-\langle g_{i_1},g_{i_2} \rangle_{L^2(\mu)})^2  \\
S_2 &=& \sum_{j_1,j_2=1}^{n} (C(m+j_1,m+j_2)-2\langle h_{j_1} , h_{j_2} \rangle_{L^2(\mu^2)})^2 + 4\|h_{j_1} \star_2^1 h_{j_2}  \|^2_{L^2(\mu)}
+ 8 \|h_{j_1} \star_1^1 h_{j_2}\|^2_{L^2(\mu^2)} \\
S_3 &=& \sum_{i=1}^{m} \sum_{j=1}^{n} 2C(i,m+j)^2 + 5\|g_i \star_1^1 h_{j}\|^2_{L^2(\mu)} \\
S_4 &=& m^2\sum_{i=1}^{m} \|g_i\|^3_{L^3(\mu)} + 8n^2 \sum_{j=1}^{n} \|h_j\|_{L^2(\mu^2)} (\|h_j\|^2_{L^4(\mu^2)}+\sqrt{2}\|h_{j_1} \star_1^0 h_{j_1}  \|_{L^2(\mu^3)}  )  \\
S_5 &=& \sum_{j_1,j_2=1}^{n} (C(m+j_1,m+j_2)-2\langle h_{j_1} , h_{j_2} \rangle_{L^2(\mu^2)})^2
+ 4 \|h_{j_1} \star_1^0 h_{j_1}  \|_{L^2(\mu^3)} \times
\|h_{j_2} \star_1^0 h_{j_2}  \|_{L^2(\mu^3)}  \\
& &+ 8 \|h_{j_1} \star_1^1 h_{j_1}  \|_{L^2(\mu^2)} \times \|h_{j_2} \star_1^1 h_{j_2}  \|_{L^2(\mu^2)} \\
S_6 &=& \sum_{i=1}^{m} \sum_{j=1}^{n} 2C(i,m+j)^2 + 5 \|g_{i}\|^2_{L^2(\mu)} \times \|h_{j} \star_1^1 h_{j}  \|_{L^2(\mu^2)} \\
\end{eqnarray*}
\end{proposition}

\begin{proof}
Assumptions 1 and 2 in the statement ensure that each integral appearing in the proof is well-defined, and that the use of Fubini arguments is justified. In view of Theorem \ref{ineq_tala}, our strategy is to study the quantities in line (\ref{unotrino}) and line (\ref{unotrino2}) separately. On the one
hand, we know that:
for $1\leq i\leq m, 1\leq j \leq n$,
$$ D_z I_1(g_i(\cdot)) = g_i(z) , \qquad -D_z L^{-1} I_1(g_i(\cdot)) =g_i(z) $$
$$ D_z I_2(h_j(\cdot,\cdot)) = 2I_1(h_j(z,\cdot)) , \qquad -D_z L^{-1} I_2(h_j(\cdot,\cdot)) = I_1(h_j(z,\cdot)) $$

\noindent Then, for any given constant $a$, we have:
\begin{description}
\item[--] for $1\leq i\leq m, 1\leq j \leq n$,
$$ \E[(a-\langle D_z I_1(g_{i_1}) ,-D_z L^{-1} I_1(g_{i_2}) \rangle)^2] = (a-\langle g_{i_1},g_{i_2} \rangle_{L^2(\mu)})^2 ;$$
\item[--] for $ 1\leq j_1,j_2 \leq n $,
\begin{eqnarray*}
& & \E[(a-\langle D_z I_2(h_{j_1}),-D_z L^{-1} I_2(h_{j_2}) \rangle)^2] \\
&=& (a-2\langle h_{j_1} , h_{j_2} \rangle_{L^2(\mu^2)})^2 + 4\|h_{j_1} \star_2^1 h_{j_2}  \|^2_{L^2(\mu)}
+ 8 \|h_{j_1} \star_1^1 h_{j_2}\|^2_{L^2(\mu^2)} ;
\end{eqnarray*}
\item[--] for $  1\leq i \leq m,  1\leq j \leq n$,
$$\E[(a-\langle D_z I_2(h_j) ,-D_z L^{-1} I_1(g_i) \rangle)^2]
=a^2 + 4\|g_i \star_1^1 h_{j}\|^2_{L^2(\mu)}$$
$$\E[(a-\langle D_z I_1(g_i) ,-D_z L^{-1} I_2(h_j) \rangle)^2]
=a^2 + \|g_i \star_1^1 h_{j}\|^2_{L^2(\mu)} .$$
\end{description}
So
$$
(\ref{unotrino})= \cfrac{1}{2}
\sqrt{S_1 + S_2 + S_3}
$$
where
$S_1 , S_2 , S_3$ are defined as in the statement of proposition. \\

\noindent On the other hand,
$$\left(\sum_{i=1}^{2}|D_z F_i | \right)^2 = \left(\sum_{i=1}^{m} |g_i(z)|+
2 \sum_{j=1}^{n} |I_1(h_j(z,\cdot))| \right)^2, $$
$$ \smd|D_z L^{-1} F_i|  = \sum_{i=1}^{m} |g_i(z)|+ \sum_{j=1}^{n} |I_1(h_j(z,\cdot))| .$$
As the following inequality holds for all positive reals $a,b$:
$$(a+2b)^2(a+b) \leq (a+2b)^3 \leq 4a^3+32b^3 ,$$
we have,
\begin{eqnarray*}
& & \E\left[\left(\smd|D_z F_i | \right)^2 \left(\smd|D_z L^{-1} F_i |  \right) \right] \\
&=& \E\left[\left(\sum_{i=1}^{m} |g_i(z)|+
2 \sum_{j=1}^{n} |I_1(h_j(z,\cdot))|\right)^2
 \left(\sum_{i=1}^{m} |g_i(z)|+ \sum_{j=1}^{n} |I_1(h_j(z,\cdot))| \right)  \right]
\\
&\leq& \E \left[4 \left(\sum_{i=1}^{m} |g_i(z)| \right)^3+
32 \left(\sum_{j=1}^{n} |I_1(h_j(z,\cdot))| \right)^3 \right] \\
&\leq& \E[4m^2 \sum_{i=1}^{m} |g_i(z)|^3+
32n^2 \sum_{j=1}^{n} |I_1(h_j(z,\cdot))|^3] .
\end{eqnarray*}
By applying the Cauchy-Schwarz inequality, one infers that
$$\int_Z \mu(dz)\E[|I_1(h(z,\cdot))|^3] \leq \sqrt{\E \left[\int_Z \mu(dz) |I_1(h(z,\cdot))|^4 \right] }
\times \|h\|_{L^2(\mu^2)} .$$
 Notice that
 $$ \E \left[\int_Z \mu(dz) |I_1(h(z,\cdot))|^4 \right] = 2\|h\star^1_2h\|^2_{L^2(\mu)} + \|h\|^4_{L^4(\mu^2)} $$
 We have
\begin{eqnarray*}
(\ref{unotrino2})&=& \cfrac{1}{4}m^2 \|C^{-1}\|^{3/2}_{op}
\|C\|_{op} \int_Z \mu(dz)\E\left[\left(\smd|D_z
F_i | \right)^2 \left(\smd|D_z L^{-1} F_i |  \right) \right] \\
&\leq&  \|C^{-1}\|^{3/2}_{op} \|C\|_{op} \Big(m^2\sum_{i=1}^{m} \|g_i\|^3_{L^3(\mu)} \\
& &+ 8n^2 \sum_{j=1}^{n} \|h_j\|_{L^2(\mu^2)} (\|h_j\|^2_{L^4(\mu^2)}+\sqrt{2}\|h_j\star_2^1 h_j\|_{L^2(\mu)}) \Big) \\
&=&  \|C^{-1}\|^{3/2}_{op} \|C\|_{op}  S_4
\end{eqnarray*}
We will now apply Lemma \ref{lemma_ineq} to further assess some of the summands appearing the definition of $S_2$,$S_3$. Indeed,
\begin{description}
\item[--] for $1\leq j_1,j_2 \leq n $,
$$\|h_{j_1} \star_2^1 h_{j_2}  \|^2_{L^2(\mu)}   \leq   \|h_{j_1} \star_1^0 h_{j_1}  \|_{L^2(\mu^3)} \times \|h_{j_2} \star_1^0 h_{j_2}  \|_{L^2(\mu^3)}    $$
$$ \|h_{j_1} \star_1^1 h_{j_2}\|^2_{L^2(\mu^2)}   \leq  \|h_{j_1} \star_1^1 h_{j_1}  \|_{L^2(\mu^2)} \times \|h_{j_2} \star_1^1 h_{j_2}  \|_{L^2(\mu^2)}  ; $$
\item[--] for $  1\leq i \leq m, 1\leq j \leq n$,
$$ \|g_i \star_1^1 h_{j}\|^2_{L^2(\mu)}   \leq  \|g_{i}\|^2_{L^2(\mu)} \times \|h_{j} \star_1^1 h_{j}  \|_{L^2(\mu^2)}   $$
by using the equality $\|g_{i}^{(k)}\star_0^0 g_{i}^{(k)}\|^2_{L^2(\mu^2)} = \|g_{i}^{(k)}\|^4_{L^2(\mu)} $.\\
\end{description}
Consequently,
\begin{eqnarray*}
S_2 &\leq& \sum_{j_1,j_2=1}^{n} (C(m+j_1,m+j_2)-2\langle h_{j_1} , h_{j_2} \rangle_{L^2(\mu^2)})^2
+ 4\|h_{j_1} \star_1^0 h_{j_1}  \|_{L^2(\mu^3)}
\times \|h_{j_2} \star_1^0 h_{j_2}  \|_{L^2(\mu^3)}  \\
& &+ 8\|h_{j_1} \star_1^1 h_{j_1}  \|_{L^2(\mu^2)} \times \|h_{j_2} \star_1^1 h_{j_2}  \|_{L^2(\mu^2)} \\
&=& S_5 , \\
S_3 &\leq&  \sum_{i=1}^{m} \sum_{j=1}^{n} 2C(i,m+j)^2 + 5\|g_{i}\|^2_{L^2(\mu)}\times \|h_{j} \star_1^1 h_{j}  \|_{L^2(\mu^2)} \\
&=& S_6
\end{eqnarray*}
\end{proof}\\

\begin{remark} \rm{
If the matrix $C$ is positive definite, then we have
\begin{eqnarray*}
d_{2}(F,X) &\leq&   \|C^{-1}\|_{op} \|C\|_{op}^{1/2} \sqrt{S_1 + S_2 + S_3}  +   \cfrac{\sqrt{2\pi}}{2}  \|C^{-1}\|^{3/2}_{op} \|C\|_{op}  S_4  \\
&\leq&    \|C^{-1}\|_{op} \|C\|_{op}^{1/2} \sqrt{S_1 + S_5 + S_6}  +   \cfrac{\sqrt{2\pi}}{2}  \|C^{-1}\|^{3/2}_{op} \|C\|_{op}  S_4
\end{eqnarray*}
by using Theorem \ref{ineq_2W}.
}
\end{remark}

The following result can be proved by means of Proposition \ref{example12}, or as a particular case of Theorem \ref{CLT}.
\begin{corollary} \label{CLT_cor}
Let $d=m+n$, with $m,n\geq 1$ two integers . Let $X_C \sim \N_d(0,C)$ be a centered $d$-dimensional Gaussian vector, where $C=\{C(s,t):s,t = 1,\ldots,d \}$ is a $d\times d$  nonnegative definite matrix such that $$C(i,j+m)=0,\qquad \forall 1\leq i \leq m, 1\leq j \leq n. $$
Assume that
$$F^{(k)}=(F_1^{(k)},\ldots,F_d^{(k)}) := (I_1(g_1^{(k)}),\ldots, I_1(g_m^{(k)}),I_2(h_1^{(k)}),\ldots, I_2(h_n^{(k)}))$$
where for all $k$, the kernels $g_1^{(k)},\ldots,g_m^{(k)}$ and
$h_1^{(k)},\ldots,h_n^{(k)} $ satisfy respectively the technical Conditions 1 and 2 in Proposition \ref{example12} .  Assume moreover that the following conditions hold for each $k\geq 1$:
\begin{enumerate}
  \item
  $$\lim_{k\rightarrow \infty} \E[F^{(k)}_s F^{(k)}_t ]=C(s,t),\qquad 1\leq s,t \leq d.$$\\
      or equivalently
      $$ \lim_{k\rightarrow \infty} \langle g^{(k)}_{i_1},g^{(k)}_{i_2} \rangle_{L^2(\mu)} = C(i_1,i_2),
      \qquad 1\leq i_1,i_2 \leq m,$$
      $$ \lim_{k\rightarrow \infty} 2\langle h^{(k)}_{j_1} , h^{(k)}_{j_2} \rangle_{L^2(\mu^2)} = C(m+j_1,m+j_2),
      \qquad 1\leq j_1,j_2 \leq n .$$
  \item For every $i=1,\ldots,m$ and every $j=1,\ldots, n$, one has
the following conditions are satisfied as $k\rightarrow \infty$: \\
(a) $\|g_i^{(k)}\|^3_{L^3(\mu)} \rightarrow 0$; (b) $\|h_j^{(k)}\|^2_{L^4(\mu^2)} \rightarrow 0$;\\
(c) $\|h_j^{(k)}\star_2^1 h_j^{(k)}\|_{L^2(\mu)}
= \|h_j^{(k)} \star_1^0 h_j^{(k)}  \|_{L^2(\mu^3)} \rightarrow 0$; \\
(d) $\|h_{j}^{(k)} \star_1^1 h_{j}^{(k)}  \|^2_{L^2(\mu^2)}
   \rightarrow 0$.
\end{enumerate}
Then $ F^{(k)} \rightarrow X $ in law, as $k\rightarrow \infty$. An explicit bound on the speed of convergence in the distance $d_3$ is provided by Proposition \ref{example12}.
\end{corollary}

\end{subsection}

\begin{subsection}{Vector of functionals
of Ornstein-Uhlenbeck processes} \label{section_OU_ex}
In this section, we study CLTs for some functionals of Ornstein-Uhlenbeck Lévy process. These processes have been intensively studied in recent years, and applied to various domains such as e.g. mathematical finance (see \cite{bnshep}) and non-parametric Bayesian survival analysis (see e.g. \cite{bpp, pp}). Our results are multi-dimensional generalizations of the content of \cite[Section 7]{pstu} and \cite[Section 4]{pt}.

We denote by $\hat{N}$ a centered Poisson measure over $\R \times \R$, with control measure given by
$\nu(du) $, where $\nu(\cdot)$ is positive, non-atomic and $\sigma$-finite. For all positive real number $\lambda$, we define the stationary \emph{Ornstein-Uhlenbeck Lévy process} with parameter $\lambda$ as
$$ Y_t^{\lambda} = I_1(f_t^{\lambda}) = \sqrt{2\lambda} \int_{-\infty}^{t} \int_{\R} u \exp(-\lambda(t-x)) \hat{N}(du,dx) ,\quad t\geq 0$$
where $f_t^{\lambda}(u,x) = \sqrt{2\lambda} \1_{( -\infty,t]}(x) u \exp(-\lambda(t-x)) $. We make the following technical assumptions on the measure $\nu$: $\int_{\R} u^j \nu(du)<\infty $ for $j=2,3,4,6$, and $\int_{\R} u^2 \nu(du)=1 $, to ensure among other things that $Y_t^{\lambda}$ is well-defined. These assumptions yield in particular that
$$ \textbf{Var}(Y_t^{\lambda}) =\E[(Y_t^{\lambda})^2] =2 \lambda \int_{-\infty}^{t} \int_{\R} u^2 \exp(-2\lambda(t-x)) \nu(du)dx = 1
$$
We shall obtain Central Limit Theorems for three kind of functionals of Ornstein-Uhlenbeck Lévy processes. In particular, each of the forthcoming examples corresponds to a ``realized empirical moment'' (in continuous time) associated with $Y^{\lambda}$, namely: Example 1 corresponds to an asymptotic study of the mean, Example 2 concerns second moments, whereas Example 3 focuses on joint second moments of shifted processes.\\

Observe that all kernels considered in the rest of this section automatically satisfy our Assumptions A, B and C.\\

\noindent\textbf{Example 1 (Empirical Means)} \\
We first recall the definition of \textbf{Wasserstein} distance.
\begin{definition}
 The {\bf Wasserstein distance} between the laws of
  two $\Rd$-valued random vectors $X$ and $Y$ with $\E \|X\|_{\Rd}$,$\E \|Y\|_{\Rd}<\infty$,  written $d_{w}(X,Y)$,
  is given by
  $$ d_{w}(X,Y) = \sup_{g\in \mathcal{H}}|\E[g(X)]-\E[g(Y)]|, $$
  where $\mathcal{H}$ indicates the collection of all functions $g\in \C^1(\Rd)$
  such that $\|g\|_{Lip}\leq 1$.
\end{definition}

We define the functional $A(T,\lambda)$ by $A(T,\lambda)= \cfrac{1}{\sqrt{T}} \int_0^T Y_t^{\lambda} dt $. We recall the following limit theorem for $A(T,\lambda)$ , taken from Example 3.6 in \cite{pstu}.
\begin{theorem}
As $ T \rightarrow \infty $,
$$ \cfrac{A(T,\lambda)}{\sqrt{2/\lambda}}= \cfrac{1}{\sqrt{2T/\lambda}} \int_0^T Y_t^{\lambda} dt \overset{(law)}{\longrightarrow} X \sim \N(0,1) , $$
and there exists a constant $0<\alpha(\lambda)<\infty$, independent of $T$ and such that
$$ d_w \left(\cfrac{A(T,\lambda)}{\sqrt{2/\lambda}} , X \right)
\leq \cfrac{\alpha(\lambda)}{\sqrt{T}} .$$
\end{theorem}

Here, we present a multi-dimensional generalization of the above result.
\begin{theorem}
For $\lambda_1,\ldots,\lambda_d >0$, as $T\rightarrow \infty$,
\begin{equation}
 \bar{A}(T)=(A(T,\lambda_1),\ldots,A(T,\lambda_d)) \overset{(law)}{\longrightarrow} X_B ,
\end{equation}
where $X_B$ is a centered $d$-dimensional Gaussian vector with covariance matrix $B=(B_{ij})_{d\times d} $, with
$B_{ij}= 2/\sqrt{\lambda_i \lambda_j}\, ,\, 1\leq i,j\leq d  $.  Moreover, there exists a constant $0<\alpha=\alpha(\bar{\lambda})=\alpha(\lambda_1,\ldots,\lambda_d)<\infty$, independent of $T$ and such that
$$ d_3(\bar{A}(T) , X_B)
\leq \cfrac{\alpha(\bar{\lambda})}{\sqrt{T}} .$$
\end{theorem}
\begin{proof}
By applying Fubini theorem on $A(T,\lambda)$, we have
$$ \cfrac{1}{\sqrt{T}} \int_0^T Y_t^{\lambda} dt = I_1(g_{\lambda,T})$$
where
$$g_{\lambda,T}= \1_{(-\infty,T]}(x) u \sqrt{\frac{2\lambda}{T}}
\int_{x\vee 0}^T \exp(-\lambda (t-x)) dt$$
\begin{eqnarray*}
& &  \E[A(T,\lambda_i) A(T,\lambda_j)] \\
&=&  \int_{\R} u^2 \nu(du) \Big( \int_{-\infty}^{0} dx  \cfrac{2}{T \sqrt{\lambda_i \lambda_j}} \exp \big((\lambda_i+\lambda_j)x \big)\times \big(1-\exp(-\lambda_i T)\big)\times \big(1-\exp(-\lambda_j T) \big) \\
& &+ \int_0^T dx \cfrac{2}{T \sqrt{\lambda_i \lambda_j}} \exp \big((\lambda_i+\lambda_j)x \big) \times \big(\exp(-\lambda_i x) - \exp(-\lambda_i T) \big)\times \big(\exp(-\lambda_j x) - \exp(-\lambda_j T) \big)  \Big)  \\
&=& \cfrac{2}{T \sqrt{\lambda_i \lambda_j}} \Big(\cfrac{1}{\lambda_i+\lambda_j} \times \big(1-\exp(-\lambda_i T) \big) \times \big(1-\exp(-\lambda_j T) \big)
+ T - \cfrac{1}{ \lambda_i}\times(1 - \exp(-\lambda_i T)) \\
& &
- \cfrac{1}{ \lambda_j} \big(1 - \exp(-\lambda_j T) \big) + \cfrac{1}{\lambda_i+\lambda_j} \big(1 - \exp(-(\lambda_i+\lambda_j) T)\big)\Big) \\
& = & \cfrac{2}{\sqrt{\lambda_i \lambda_j}} +O\left(\cfrac{1}{T}\right) \qquad \text{as } T\rightarrow \infty .
\end{eqnarray*}
And we may verify that
$\|g_{\lambda,T}\|^3_{L^3(d\nu dx)}  \sim \cfrac{1}{\sqrt{T}} .$ for all $\lambda \in \R$. (See \cite{pstu} and \cite{pt} for details.) Finally, we deduce the conclusion by using Corollary \ref{Cor2}.
\end{proof}\\

\noindent\textbf{Example 2 (Empirical second moments)}\\
We are interested in the quadratic functional $Q(T,\lambda) $ given by:
$$Q(T,\lambda):= \sqrt{T}\left(\frac{1}{T} \int_0^T (Y_t^{\lambda})^2 dt -1 \right) ,\qquad T>0,\lambda>0$$
In \cite{pstu} and \cite{pt}, the authors have proved the following limit theorem for $Q(T,\lambda) $. (See Theorem 7.1 in \cite{pstu} and Proposition 7 in \cite{pt})
\begin{theorem}
For every $\lambda >0 $, as $T\rightarrow \infty$,
$$Q(T,\lambda):= \sqrt{T}\left(\frac{1}{T} \int_0^T (Y_t^{\lambda})^2 dt -1 \right)
\overset{(law)}{\longrightarrow}  \sqrt{\cfrac{2}{\lambda} + c_{\nu}^2} \times X  $$
where $X \sim \N(0,1)$ is a standard Gaussian random variable and $ c_{\nu}^2 = \int_{\R} u^4 \nu(du)$ is a constant. And there exists a constant $0<\beta(\lambda)<\infty$, independent of $T$ and such that
$$ d_w\left(\cfrac{Q(T,\lambda)}{\sqrt{\frac{2}{\lambda} + c_{\nu}^2}} , X \right)
\leq \cfrac{\beta(\lambda)}{\sqrt{T}} $$
\end{theorem}
We introduce here a multi-dimensional generalization of the above result.
\begin{theorem} \label{CLT_ex2_thm}
Given an integer $d\geq 2$. For $\lambda_1,\ldots,\lambda_d >0$, as $T\rightarrow \infty$,
\begin{equation} \label{CLT_ex1}
 \bar{Q}(T)=(Q(T,\lambda_1),\ldots,Q(T,\lambda_d)) \overset{(law)}{\longrightarrow} X_C ,
\end{equation}
where $X_C$ is a centered $d$-dimensional Gaussian vector with covariance matrix $C=(C_{ij})_{d\times d} $, defined by
$$C_{ij}= \cfrac{4}{\lambda_i + \lambda_j}+c_{\nu}^2,\qquad 1\leq i,j\leq d , $$
and $ c_{\nu}^2 = \int_{\R} u^4 \nu(du)$. And there exists a constant $0<\beta(\bar{\lambda}) = \beta(\lambda_1,\ldots,\lambda_d)<\infty$, independent of $T$ and such that
$$ d_3(\bar{Q}(T) , X_C)
\leq \cfrac{\beta(\bar{\lambda})}{\sqrt{T}} $$
\end{theorem}
\begin{proof}
For every $T>0 $ and $\lambda>0 $, we introduce the notations
\begin{eqnarray*}
H_{\lambda,T}(u,x;u',x') &=& (u\times u') \cfrac{\1_{(-\infty,T)^2}(x,x')}{T} \Big( \exp \big(\lambda(x+x') \big) \times \big(1-\exp(-2\lambda T)\big)\times \1_{(x\vee x'\leq 0)} \\
& & + \exp \big(\lambda(x+x')\big)\times \big(\exp(-2\lambda(x\vee x')) - \exp(-2\lambda T)\big)\times \1_{(x\vee x' > 0)} \Big)
\end{eqnarray*}
\begin{eqnarray*}
H^\star_{\lambda,T}(u,x) &=& u^2 \cfrac{\1_{(-\infty,T)}(x)}{T}
\Big( \exp(2\lambda x)\times \big(1-\exp(-2\lambda T)\big)\times \1_{(x\leq 0)} \\
& & + \exp(2\lambda x)\times \big(\exp(-2\lambda x) - \exp(-2\lambda T)\big)\times \1_{(x > 0)} \Big)
\end{eqnarray*}
By applying the multiplication formula (\ref{product}) and a Fubini argument, we deduce that
$$ Q(T,\lambda)= I_1(\sqrt{T} H^\star_{\lambda,T}) + I_2(\sqrt{T} H_{\lambda,T}) , $$
which is the sum of a single and a double Wiener-Itô integral. Instead of deducing the convergence for
$(Q(T,\lambda_1),\ldots,Q(T,\lambda_d)) $, we prove the stronger result:
\begin{equation} \label{CLT_ex2}
(I_1(\sqrt{T} H^\star_{\lambda_1,T}) ,\ldots, I_1(\sqrt{T} H^\star_{\lambda_d,T}),  I_2(\sqrt{T} H_{\lambda_1,T}) ,\ldots, I_2(\sqrt{T} H_{\lambda_d,T})) \overset{(law)}{\longrightarrow} X_D
\end{equation}
 as $T \rightarrow \infty$. Here, $X_D$ is a centered $2d$-dimensional Gaussian vector with covariance matrix $D$ defined as:
$$ D(i,j)=\left\{
            \begin{array}{ll}
              c_{\nu}^2, & \hbox{if } 1\leq i,j\leq d \\
              \cfrac{4}{\lambda_i + \lambda_j}, & \hbox{if } d+1\leq i,j\leq 2d \\
              0, & \hbox{otherwise.} \\
            \end{array}
          \right.
$$
We prove (\ref{CLT_ex2}) in two steps (by using Corollary \ref{CLT_cor}). Firstly, we aim at verifying
$$\lim_{T\rightarrow \infty} \E[F^{(T)}_i F^{(T)}_j ]=D(i,j),\qquad 1\leq i,j \leq 2d,$$
for
$$ F^{(T)}_k =\left\{
            \begin{array}{ll}
              I_1(\sqrt{T} H^\star_{\lambda_k,T}) , & \hbox{if } 1\leq k\leq d \\
              I_2(\sqrt{T} H_{\lambda_k,T}), & \hbox{if } d+1\leq k\leq 2d \\
            \end{array}
          \right.
$$
Indeed, by standard calculations, we have
\begin{eqnarray*}
& & T \int_{\R\times \R} H^\star_{\lambda_i,T}(u,x) H^\star_{\lambda_j,T}(u,x) \nu(du) dx\\
&=& \cfrac{1}{T} c_{\nu}^2 \Big(\cfrac{1}{2(\lambda_i+\lambda_j)}\times \big(1-\exp(-2\lambda_i T)\big) \times \big(1-\exp(-2\lambda_j T)\big)
+ T - \cfrac{1}{2 \lambda_i} \times \big(1 - \exp(-2\lambda_i T)\big) \\
& &
- \cfrac{1}{2 \lambda_j}\times \big(1 - \exp(-2\lambda_j T)\big) + \cfrac{1}{2(\lambda_i+\lambda_j)}\times \big(1 - \exp(-2(\lambda_i+\lambda_j) T)\big)\Big) \\
& = & c_{\nu}^2 + O \left(\cfrac{1}{T} \right), \qquad \text{as } T\rightarrow \infty ,
\end{eqnarray*}
and
\begin{eqnarray*}
& & 2T \int_{\R^4} H_{\lambda_i,T}(u,x;u',x') H_{\lambda_j,T}(u,x;u',x') \nu(du)\nu(du') dxdx'\\
&=& \cfrac{2}{T} \Big( \cfrac{(1-\exp(-2\lambda_i T))\times (1-\exp(-2\lambda_j T))}{(\lambda_i + \lambda_j)^2} + \cfrac{2}{\lambda_i + \lambda_j} \times \big(T-\cfrac{1}{2\lambda_i} \big( 1-\exp(-2\lambda_i T) \big) \\
& & -\cfrac{1}{2\lambda_j} \times \big( 1-\exp(-2\lambda_j T)\big) +\cfrac{1}{2(\lambda_i + \lambda_j)} \times \big( 1-\exp(-2(\lambda_i + \lambda_j) T)\big) \big)  \Big)
\\
& = & \cfrac{4}{\lambda_i + \lambda_j} + O \left(\cfrac{1}{T} \right), \qquad \text{as } T\rightarrow \infty .
\end{eqnarray*}
\noindent Secondly, we use the fact that for $\lambda=\lambda_1,\ldots,
\lambda_d $, the following asymptotic relations holds as $T\rightarrow \infty$:\\
$(a)\qquad \|\sqrt{T} H^\star_{\lambda,T}\|^3_{L^3(d\nu dx)}  \sim \cfrac{1}{\sqrt{T}}\,;$\\
$(b)\qquad \|\sqrt{T} H_{\lambda,T}\|^2_{L^4((d\nu dx)^2)} \sim \cfrac{1}{\sqrt{T}}\,;$\\
$(c)\qquad \|(\sqrt{T} H_{\lambda,T}) \star_2^1 (\sqrt{T} H_{\lambda,T})\|_{L^2(d\nu dx)} = \|(\sqrt{T} H_{\lambda,T}) \star_1^0 (\sqrt{T} H_{\lambda,T})  \|_{L^2((d\nu dx)^3)} \sim \cfrac{1}{\sqrt{T}}\,;$\\
$(d)\qquad \|(\sqrt{T} H_{\lambda,T}) \star_1^1 (\sqrt{T} H_{\lambda,T})  \|_{L^2((d\nu dx)^2)}
   \sim \cfrac{1}{\sqrt{T}}\,;$\\
$(e)\qquad \|(\sqrt{T} H^\star_{\lambda,T}) \star_1^1 (\sqrt{T} H_{\lambda,T})\|_{L^2(d\nu dx)} \sim \cfrac{1}{\sqrt{T}}\,.$  \\
The reader is referred to \cite[Section 7]{pstu} and \cite[Section 4]{pt} for a proof of the above asymptotic relations.
\end{proof} \\

\noindent\textbf{Example 3 (Empirical joint moments of shifted processes)} \\
We are now able to study a generalization of Example 2. We define
$$Q_h(T,\lambda):= \sqrt{T} \left(\frac{1}{T} \int_0^T Y_t^{\lambda}Y_{t+h}^{\lambda} dt - \exp(-\lambda h) \right) ,\qquad h> 0, T>0, \lambda>0 .$$
The theorem below is a multi-dimensional CLT for $Q_h(T,\lambda)$.
\begin{theorem}
For $\lambda_1,\ldots,\lambda_d >0$ and $h\geq 0$, as $T\rightarrow \infty$,
\begin{equation} \label{CLT_ex3}
 \bar{Q}_h (T)=(Q_h(T,\lambda_1),\ldots, Q_h(T,\lambda_d)) \overset{(law)}{\longrightarrow} X_E ,
\end{equation}
where $X_E$ is a centered $d$-dimensional Gaussian vector with covariance matrix $E=(E_{ij})_{d\times d} $, with
$$E_{ij}= \cfrac{4}{\lambda_i + \lambda_j}+c_{\nu}^2 \exp \Big(-(\lambda_i+\lambda_j)h \Big) ,\qquad 1\leq i,j\leq d $$ and $ c_{\nu}^2 = \int_{\R} u^4 \nu(du)$. Moreover, there exists a constant $0<\gamma(h,\bar{\lambda}) = \gamma(h,\lambda_1,\ldots,\lambda_d)<\infty$, independent of $T$ and such that
$$ d_3(\bar{Q}_h (T) , X_E)
\leq \cfrac{\gamma(h,\bar{\lambda})}{\sqrt{T}} $$
\end{theorem}
\begin{proof}
We have
\begin{eqnarray*}
\int_0^{T} Y_t^{\lambda} Y_{t+h}^{\lambda} dt &=& \int_0^{T-h} I_1(f_t^{\lambda}) I_1(f_{t+h}^{\lambda}) dt \\
&=&  \int_0^{T-h} \Big( I_2(f_t^{\lambda} \star^0_0 f_{t+h}^{\lambda}) + I_1(f_t^{\lambda} \star^0_1 f_{t+h}^{\lambda}) + f_t^{\lambda} \star^1_1 f_{t+h}^{\lambda} \Big)dt \\
&=&  \int_0^{T-h} \Big(I_2(\hat{h}_{t,h}^{\lambda}  ) + I_1( \hat{h}_{t,h}^{*,\lambda} ) + C \Big) dt \\
&=& I_2(T H_{\lambda,T}^{h}  ) + I_1(T H_{\lambda,T}^{*,h} ) + C(T-h)
\end{eqnarray*}
 by using multiplication formula (\ref{product}) and Fubini theorem. By simple calculations, we obtain that
$$ \hat{h}_{t,h}^{\lambda} (u,x;u',x') = 2\lambda \1_{(-\infty,t] \times (-\infty,t+h]} (x,x') \times uu' \exp(-\lambda (2t+h-x-x')) $$
$$ \hat{h}_{t,h}^{*,\lambda} (u,x) = 2\lambda \1_{(-\infty,t]} (x) \times u^2 \exp(-\lambda (2t+h-2x)) $$
as well as
\begin{eqnarray*}
H_{\lambda,T}^{*,h} (u,x) &=& \cfrac{1}{T} \int_0^T  \hat{h}_{t,h}^{*,\lambda} (u,x) dt \\
&=& u^2 \cfrac{\1_{(-\infty, T]} (x) }{T}\times \exp(\lambda(2x-h)) \times
\Big(\1_{\{x>0\}}\times (\exp(-2 \lambda x)-\exp(-2\lambda T))\\
& & +  \1_{\{x\leq 0\}} \times(1
- \exp(-2\lambda T))  \Big)
\end{eqnarray*}
\begin{eqnarray*}
H_{\lambda,T}^{h} (u,x;u',x') &=& \cfrac{1}{T} \int_0^T  \hat{h}_{t,h}^{\lambda} (u,x;u',x') dt \\
&=&  u u' \cfrac{\1_{(-\infty, T]} (x) \1_{(-\infty, T+h]} (x') }{T} \times \exp(\lambda(x+x'-h)) \\
& &\times \Big(\1_{\{x\vee (x'-h)>0\}} \times \big(\exp(-2\lambda (x\vee (x'-h)) ) -\exp(-2\lambda T)\big)\\
& & +  \1_{\{x\vee (x'-h)\leq 0\}} \times (1- \exp(-2\lambda T ))  \Big)
\end{eqnarray*}

Similar to the procedures in the precedent example, we prove the stronger result:
\begin{equation}
(I_1(\sqrt{T} H^{\star,h}_{\lambda_1,T}) ,\ldots, I_1(\sqrt{T} H^{\star,h}_{\lambda_d,T}),  I_2(\sqrt{T} H^{h}_{\lambda_1,T}) ,\ldots, I_2(\sqrt{T} H^{h}_{\lambda_d,T})) \overset{(law)}{\longrightarrow} X_{D^h}
\end{equation}
 as $T \rightarrow \infty$. Here, $X_{D^h}$ is a centered $2d$-dimensional Gaussian vector with covariance matrix $D^h$ defined as:
$$ D^h(i,j)=\left\{
            \begin{array}{ll}
              c_{\nu}^2 \exp(-(\lambda_i + \lambda_j)h), & \hbox{if } 1\leq i,j\leq d \\
              \cfrac{4}{\lambda_i + \lambda_j}, & \hbox{if } d+1\leq i,j\leq 2d \\
              0, & \hbox{otherwise.} \\
            \end{array}
          \right.
$$

We have
\begin{eqnarray*}
& & T \int_{\R\times \R} H_{\lambda,T}^{*,h} (u,x) H_{\lambda,T}^{*,h} (u,x) \nu(du) dx\\
&=& \cfrac{1}{T} c_{\nu}^2 \Big( \int_{-\infty}^{0} dx
\exp\big((\lambda_i+\lambda_j)(2x-h)\big) \times \big(1-\exp(-2\lambda_i T)\big) \times \big(1-\exp(-2\lambda_j T)\big) \\
 &+& \int_0^T dx \exp\big((\lambda_i+\lambda_j)(2x-h)\big) \times \big(\exp(-2\lambda_i x) - \exp(-2\lambda_i T)\big) \times \big(\exp(-2\lambda_j x) - \exp(-2\lambda_j T)\big) \Big)  \\
& = & c_{\nu}^2 \exp(-(\lambda_i+\lambda_j)h) +O \left(\cfrac{1}{T} \right) .\qquad \text{as } T\rightarrow \infty ,
\end{eqnarray*}
We notice that
$$H_{\lambda,T}^{h} (u,x;u',x') = H_{\lambda,T} (u,x;u',x'-h)  $$
Then, as shown in the proof of Theorem \ref{CLT_ex2_thm}, we have
$$
2T \int_{\R\times \R} H_{\lambda,T}^{h} (u,x) H_{\lambda,T}^{h} (u,x) \nu(du) dx =  \cfrac{4}{\lambda_i + \lambda_j} +O \left(\cfrac{1}{T} \right) .\qquad \text{as } T\rightarrow \infty .
$$
Just as the precedent example, we may verify that for $\lambda=\lambda_1,\ldots,
\lambda_d $ and $h\geq 0$, the following asymptotic relations holds as $T\rightarrow \infty$:\\
$(a)\qquad \|\sqrt{T} H_{\lambda,T}^{*,h}\|^3_{L^3(d\nu dx)}  \sim \cfrac{1}{\sqrt{T}}\,;$\\
$(b)\qquad \|\sqrt{T} H_{\lambda,T}^{h}\|^2_{L^4((d\nu dx)^2)} \sim \cfrac{1}{\sqrt{T}}\,;$\\
$(c)\qquad \|(\sqrt{T} H_{\lambda,T}^{h}) \star_2^1 (\sqrt{T} H_{\lambda,T}^{h}\|_{L^2(d\nu dx)} = \|(\sqrt{T} H_{\lambda,T}) \star_1^0 (\sqrt{T} H_{\lambda,T}^{h})  \|_{L^2((d\nu dx)^3)} \sim \cfrac{1}{\sqrt{T}}\,;$\\
$(d)\qquad \|(\sqrt{T} H_{\lambda,T}^{h}) \star_1^1 (\sqrt{T} H_{\lambda,T}^{h})  \|_{L^2((d\nu dx)^2)}
   \sim \cfrac{1}{\sqrt{T}}\,;$\\
$(e)\qquad \|(\sqrt{T} H^{*,h}_{\lambda,T}) \star_1^1 (\sqrt{T}^h H_{\lambda,T})\|_{L^2(d\nu dx)} \sim \cfrac{1}{\sqrt{T}}\,.$  \\

We conclude the proof by analogous arguments as in the proof of (\ref{CLT_ex1}).
\end{proof}\\

The calculations above enable us to derive immediately the following new one-dimensional result, which is a direct generalization of Theorem 5.1 in \cite{pstu}.

\begin{corollary}
For every $\lambda >0 $, as $T\rightarrow \infty$,
$$Q_h(T,\lambda) \overset{(law)}{\longrightarrow}  \sqrt{\cfrac{2}{\lambda} + c_{\nu}^2 \exp(-2\lambda h) }  \times X  $$
where $X \sim \N(0,1)$ is a standard Gaussian random variable. Moreover, there exists a constant $0<\gamma(h,\lambda)<\infty$, independent of $T$ and such that
$$ d_w \left(\cfrac{Q_h(T,\lambda)}{\sqrt{2/\lambda + c_{\nu}^2 \exp(-2\lambda h) } } \, , X \right)
\leq \cfrac{\gamma(h,\lambda)}{\sqrt{T}} $$
\end{corollary}
\end{subsection}

\end{section}

\bibliographystyle{plain}

\begin{section}{Appendix: Malliavin operators on the Poisson space}\label{APPENDIX}
We now define some Malliavin-type operators associated with a Poisson measure $\hat{N}$, on the Borel space $(Z,\mathcal{Z})$, with non-atomic control measure $\mu$. We follow the work by Nualart and Vives \cite{nuaviv}, which is in turn based on the classic definition of Malliavin operators on the Gaussian space (see e.g. \cite{Malliavin, nualart}).\\


\noindent \textbf{(I) The derivative operator $D$}. \\
For every $F\in L^2(\p)$, the derivative of $F$, $DF$ is defined as an element of $L^2(\p;L^2(\mu))$, that is, of the space of the jointly measurable random functions $u:\Omega \times Z \mapsto \R$ such that $\E \left[\int_Z u_z^2 \mu(dz) \right] <\infty$.
\begin{definition}
 \begin{enumerate}
   \item The domain of the derivative operator $D$, written  ${\rm dom} D$, is the set of all random variables $F\in L^2(\p)$ admitting a chaotic decomposition (\ref{chao}) such that
$$ \sum_{k\geq 1} k k!\|f_k \|^2_{L^2(\mu^k)} < \infty ,$$
   \item For any $F\in {\rm dom}D$, the random function $z \mapsto D_z F$ is defined by
$$ D_z F= \sum_{k \geq 1}^{\infty} k I_{k-1}(f_k(z,\cdot)) .$$
 \end{enumerate}
\end{definition}

\noindent \textbf{(II) The divergence operator $\delta$}.  \\
Thanks to the chaotic representation property of $\hat{N}$, every random function
$u \in L^2(\p,L^2(\mu))$ admits a unique representation of the type
\begin{equation} \label{skor}
 u_z = \sum_{k \geq 0}^{\infty}  I_{k}(f_k(z,\cdot)) ,\,\, z\in Z,
\end{equation}
where the kernel $f_k$ is a function of $k+1$ variables, and $f_k(z,\cdot)$ is an element of $L^2_s(\mu^k)$. The {\sl divergence operator} $\delta(u)$ maps a random function $u$ in its domain to an element of $L^2(\p)$.\\

\begin{definition}
\begin{enumerate}
  \item  The domain of the divergence operator, denoted by  ${\rm dom} \delta$, is the collection of all $u\in L^2(\p,L^2(\mu))$ having the above chaotic expansion (\ref{skor}) satisfied the condition:
$$ \sum_{k\geq 0}  (k+1)! \|f_k \|^2_{L^2(\mu^(k+1))} < \infty. $$
  \item For $u\in {\rm dom}\delta$, the random variable $\delta(u)$ is given by
      $$ \delta (u) = \sum_{k\geq 0} I_{k+1}(\tilde{f}_k), $$
      where $\tilde{f}_k$ is the canonical symmetrization of the $k+1$ variables function $f_k$.
\end{enumerate}
\end{definition}
As made clear in the following statement, the operator $\delta$ is indeed the adjoint operator of $D$.
\begin{lemma}[Integration by parts]\label{L : IBP}
 For every $G\in {\rm dom} D$ and $u\in {\rm dom} \delta$, one has that
$$ \E[G \delta(u)] = \E[\langle D G, u \rangle_{L^2(\mu)}]. $$
\end{lemma}
The proof of Lemma \ref{L : IBP} is detailed e.g. in \cite{nuaviv}.\\

\noindent \textbf{(III) The Ornstein-Uhlenbeck generator $L$}.
\begin{definition}
\begin{enumerate}
  \item  The domain of the Ornstein-Uhlenbeck generator, denoted by  ${\rm dom} L$, is the collection of all $F \in L^2(\p)$ whose chaotic representation \label{chao} verifies the condition:
$$ \sum_{k\geq 1}  k^2 k! \|f_k \|^2_{L^2(\mu^k)} < \infty $$
  \item The Ornstein-Uhlenbeck generator $L$ acts on random variable $F\in {\rm dom}L$ as follows:
      $$ LF = - \sum_{k\geq 1} k I_{k}(f_k) .$$
\end{enumerate}
\end{definition}

\medskip

\noindent \textbf{(IV) The pseudo-inverse of $L$}.
\begin{definition}
\begin{enumerate}
  \item  The domain of the pseudo-inverse of the Ornstein-Uhlenbeck generator, denoted by $L^{-1}$, is the space $L^2_0(\p)$ of \emph{centered} random variables in $L^2(\p)$.
  \item For $F = \sum\limits_{k\geq 1} I_k (f_k) \in L^2_0(\p)$ , we set
      $$ L^{-1}F = - \sum_{k\geq 1} \cfrac{1}{k} I_{k}(f_k). $$
\end{enumerate}
\end{definition}
\end{section}
\end{document}